\documentclass[12pt,a4paper]{amsart}
\usepackage[margin = 2cm]{geometry}
\usepackage{amsmath,amssymb,amsthm}
\usepackage{mathrsfs}
\usepackage{multirow}
\usepackage{diagbox}
\usepackage{graphicx}
\usepackage{color}
\usepackage{hyperref}
\usepackage{bbm}
\usepackage{mathabx}
\usepackage{natbib}

\newtheorem{theorem}{Theorem}[section]
\newtheorem{corollary}[theorem]{Corollary}
\newtheorem{assumption}[theorem]{Assumption}
\newtheorem{lemma}[theorem]{Lemma}
\newtheorem{proposition}[theorem]{Proposition}

\newtheoremstyle{definition}{}{}{}{}{\bfseries}{}{ }{}
\theoremstyle{definition}
\newtheorem{definition}[theorem]{Definition}

\theoremstyle{remark}

\newtheorem{remark}{Remark}

\def\cdf{\F}
\def\estcdf{\hF}

\def\dens{f_\varepsilon}

\def\resk{\hat{\varepsilon}_\mathbf{k}}
\def\resestk{\hat{\varepsilon}_\mathbf{k}}
\def\k{\mathbf{k}}
\def\K{\mathbf{K}}

\def\funk{g}
\def\est{\hat{g}}
\def\R{\mathcal{R}}
\def\desk{z_\mathbf{k}}
\def\weight{w_\mathbf{k}}

\def\1{\mathbf{1}}

\def\F{\mathbb{F}}
\def\hF{\hat{\mathbb{F}}_n}

\def\RR{\mathbb{R}}

\def\lb{\left(}
\def\rb{\right)}
\def\la{\left|}
\def\ra{\right|}
\def\lbr{\left[}
\def\rbr{\right]}
\def\lbp{\left<}
\def\rbp{\right>}

\def\GS{B_{(k_1, k_2)}}
\def\GSk{B_\mathbf{k}}
\def\weightk{w_\mathbf{k}}
\def\weight{w_{(k_1,k_2)}}
\def\weight{w_{\mathbf{k}}}
\def\rp{R(l,m)}
\def\rpN{\hat{R}(l,m)}

\def\smoothindex{\tau}
 \def\smoothp{\tilde{\tau}}
 \def\smoothpp{\tilde{t}}
\def\trunc{d_n} %truncation parameter for proof
\def\truncres{\varepsilon_\k^{d_n}} %truncated residual for proof
\def\smoothoc{\mathfrak{O}(v)}
\def\smoothel1{\mathcal{E}(v)}
\def\smoothel2{\mathfrak{O}(\tau,1,1)}

\numberwithin{equation}{section}

\allowdisplaybreaks

\begin{document}

\title[The empirical process of residuals from an inverse regression]
{The empirical process of residuals from an inverse regression}

% Former title.
%\title[The empirical residual process in an inverse regression problem]
%{The empirical residual process in an inverse regression problem}

% Working title
%Content of abstract at interim state

\begin{abstract}
In this paper we investigate an indirect regression model
characterized by the Radon transformation. This model is useful for
recovery of medical images obtained by computed tomography scans. The
indirect regression function is estimated using a series estimator
motivated by a spectral cut-off technique. Further, we investigate the
empirical process of residuals from this regression, and show that it
satsifies a functional central limit theorem.
\end{abstract}

\author[ T.\ Kutta,  N. \ Bissantz, J.\ Chown and H.\ Dette]
{Tim Kutta,  Nicolai Bissantz, Justin Chown and Holger Dette}

\maketitle

\bigskip

    %%%%% Section: Introduction

\section{Introduction}
\label{intro}
Computed tomography (CT) is a noninvasive imaging technique, which is
a key method for medical diagnoses.
%Due to its high rapidity and
%diagnostic conclusiveness (wir glauben ja alles, was die uns sagen ...:-)
%it has spread into diverse branches of medical applications
%e.g. oncology, dentistry and emergency surgery.
CT is based on measuring the intensity losses of X-rays sent through
a body. From these measurements an attenuation profile can be
recovered that provides an image of the body's (unobservable) interior.
The X-rays are linear and so the scanner rotates
to create a two-dimensional slice. Insight into three-dimensional
structures is obtained by considering multiple slices.
% The unit circle is used to model a detector ring surrounding the
% body, along which in practice the X-ray source rotates.
Our investigation is limited to a statistical analysis of data
gathered from a single slice. For this purpose we introduce the
inverse regression model
\begin{equation}
     Y_{\mathbf{k}}= \R    \funk (\desk)+\varepsilon_\mathbf{k},
 \quad \k \in \K ,
     \label{regressionproblem}
\end{equation}
where $(\varepsilon_\mathbf{k})_{{}\k \in \K}$ are independent
and identically distributed random variables with $ \mathbb{E} [
\varepsilon_\mathbf{k}]=0$.
Here $\K$ is a given index set, with each index $\k$ corresponding to
an X-ray path and the design point $\desk$ characterizing this path % in the \chk{unit disc} 
with associated response $Y_{\mathbf{k}}$.
Consequently, $\desk$ can be written using
coordinates $0 \leq s \leq 1$ as the distance from the origin and $0
\leq \phi \leq 2\pi$ as the angle of inclination.
The body's (true) attenuation profile along the slice is represented
by $g$, a function supported on the unit disc.
$\mathcal{R}$ is a linear operator acting on $g$ and denotes the
normalized {\it Radon transform}, i.e.\ for $0\le s \le 1$ and $0 \le
\phi \le 2\pi$,
\begin{equation} \label{radontransform}
\R\funk(s, \phi):= \frac{1}{2}(1-s^2)^{-\frac{1}{2}}
 \int\limits_{-\sqrt{1-s^2}}^{\sqrt{1-s^2}}
 \funk \big(s \cos(\phi)-t\sin(\phi),
 s \sin(\phi) + t \cos(\phi) \big)\, dt.
\end{equation}
Details on the underlying physics and applications of CT can be
found in \cite{Bu2008}.

Image reconstruction in CT is a particular case of the broad class of
linear inverse problems. An overview of the mathematical aspects of
these problems and methods to solving them can be found in the
monographs of \cite{natterer}, \cite{enghanneu1996} and
\cite{helgason2011}. Other examples of linear inverse problems are the
heat equation and convolution transforms (see \cite{mairruymgaart1996},
\cite{saitoh1997}, and \cite{cavalier2008}, among others). Additional
statistical inverse problems include errors-in-variables models and
the Berkson error model (see, for example, \cite{bishohmunruy2007},
\cite{CarDelHal2007}, \cite{KouSon2008, KouSon2009},
\cite{berbocdes2009}, \cite{kaisom2005}, \cite{delaigle2014}, and
\cite{Kato17}). The Radon transform is usually discussed in the
contexts of positron emission tomograpy (PET) and CT in medical
imaging. In the case of PET, lines-of-sight are observed along which
emissions have occurred. However, the positions of the emissions on
these lines are unknown. Here the aim is to reconstruct the emission
density (see \cite{johnstonesilbverman1990}, \cite{kortsy1993}, and
\cite{cavalier2000density}, among others). On the other hand, CT leads
to the inverse regression \eqref{regressionproblem} (see, for example,
\cite{cavalier1999Regression} and \cite{ker2010, ker2012}).

% Here, we observe the attenuation of a X-ray signal along specific
% lines-of-sight through the object of interest. This decrease of the
% X-ray intensity along several lines is observed and then used to
% reconstruct the mass density of the object of interest.

%In all these examples the first step in data analysis is the recovery
%of the quantity of interest which involves the approximate inversion
%of the corresponding operator. In most cases the relation between the
%observed  data and  the object to be  estimated is similar as in
%\eqref{regressionproblem} described via a linear operator, which is
%not well  posed, usually because of  a discontinuous inverse. The
%discontinuity of the inverse   operator (in our case
%$\mathcal{R}^{-1}$) results in  non reliable estimates if we  simply
%reconstruct $\mathcal{R}\funk$ from the data and apply
%$\mathcal{R}^{-1}$ to this estimate. Hence, additional regularization
%techniques are necessary. Early work in this context  for  CT-like
%data  can be found in  \cite{natterer, natterer1983}, where several
%regularization methods are reviewed.

We contribute to this discussion by deriving the rate of uniform,
strong consistency for a nonparametric estimator $\est$ of the unknown
function $g$ based on the popular spectral cutoff method. Further, we
derive a functional central limit theorem for the empirical process of
the resulting model residuals $\resk$, i.e.\ we investigate the
estimator
\begin{equation} \label{estimateddistributionfunction}
\hF(t) = \sum_{ \mathbf{k} \in \mathbf{K}} \weight \mathbbm{1}\left\{
    \resk \le t\right\},
 \quad t \in \RR,
\end{equation}
where the nonnegative weights $\weight$ sum to $1$ (see Section
\ref{sec3}). Statistical applications of results of this type include
validation of model assumptions. In the context of inverse regression
models, to the best of our knowledge only one result is available:
\cite{BiChDe2018}, who study an inverse regression model characterized
by a convolution transformation.

In direct regression problems, residual-based empirical
processes arising from non- and semiparametric regression estimators
have been considered by numerous authors (see \cite{AkV2013},
\cite{neumeyer2009}, \cite{mullerschickwefelmeyer2012},
\cite{Colling2016}, and \cite{Zhang2018}, among
others). \cite{DetteNeumeyerVanKeilegom2007} consider tests for a
parametric form of the variance function in a heteroscedastic
nonparametric regression by comparing the empirical distribution
function of standardized residuals calculated under a null model to
that of an alternative model. \cite{NeumeyerVanKeilegom2010} work with
a similar approach as the previous authors to propose tests for
verifying convenient forms of the regression
function. \cite{KhmaladzeKoul2009} introduce a popular distribution
free approach to addressing goodness-of-fit problems for the errors
from a nonparametric regression, where these authors introduce a
transformation of the empirical distribution function of residuals
that is useful for forming test statistics with convenient limit
distributions. All of these approaches to validating model assumptions
crucially rely on a technical asymptotic linearity property of the
residual-based empirical distribution function. We show the estimator
\eqref{estimateddistributionfunction} shares this property as well,
and the results of this article can be used immediately in approaches
to validating model assumptions in the inverse regression model
\eqref{regressionproblem} that are in the same spirit as the previously
mentioned works.

We have organized the remaining parts of the paper as follows. Model
\eqref{regressionproblem} is further discussed and we introduce the
estimator $\est$ in Section \ref{sec2}. Our main results are given in
Section \ref{sec3}. All of the proofs of our results and additional
supporting technical details may be found in the appendices.

\section{Estimation in the indirect regression model}
\label{sec2}

 In this section we give more details regarding the Radon transform model
 \eqref{regressionproblem} and introduce an estimator of the function $\funk$.

\subsection{The Radon transform}

 Following \cite{johnstonesilbverman1990} let
\begin{equation}
    \mathcal{B} := \{(r,\theta): 0 \le r \le 1, ~ 0 \le \theta \le 2 \pi \}
    \label{brainspace}
\end{equation}
 denote the unit disc, which is the two dimensional domain of the investigated
 attenuation profile $\funk$ and is called \textit{brain space} for historical reasons.
 It is equipped with the uniform distribution,  given in polar coordinates by
\begin{equation}
    d \mu (r, \theta) := \pi^{-1}r~ dr ~ d\theta. \label{mu}
\end{equation}
 This means that no prior emphasis on any region of the scanned area is given.
 The \textit{detector space} $\mathcal{D}$ is defined as
\begin{equation}
 \mathcal{D} := \{(s,\phi): 0 \le s \le 1, ~ 0 \le \phi \le 2 \pi \}
 \label{detectorspace}
\end{equation}
 with corresponding probability measure
\begin{equation}
    d\lambda(s, \phi):= 2 \pi^{-2} \sqrt{1-s^2} ~ds ~ d\phi. \label{lambda}
\end{equation}
The domain of the transformed image $\R \funk$ is $\mathcal{D}$, a parametrization of all lines (X-ray paths) crossing the unit disc. It is usally referred to as \textit{detector space}. 
 $\lambda$ is a probability measure on $\mathcal{D}$ adapted to the length of the line segments inside the disc.  
 For analytic simplicity we allow the angles in $\mathcal{B}$ and $\mathcal{D}$ to be
 exactly $0$ and $2\pi$. This is possible since the below required smoothness  of $\funk$ and
 $\mathcal{R}\funk$ entail periodicity with respect to the angular coordinates.  
 
 The Radon transform in \eqref{radontransform} defines a linear operator from $\mathcal{L}^2(\mathcal{B}, \mu) $ to
  $ \mathcal{L}^2(\mathcal{D}, \lambda)$.
 Identifying corresponding equivalence classes it can be shown that $\mathcal{R}$ is one-to-one, compact and permits a singular value
 decomposition (SVD). The SVD of $\R$ is vital for our subsequent investigations. To state it efficiently we introduce some definitions borrowed from 
 \cite{johnstonesilbverman1990}  and  \cite{BoWo1970}. Let
\begin{equation*}
    \mathcal{N}:=\big\{  (l,m): m \in \mathbb{N}_0, l=m, m-2,...,-m\big\}.
    \label{Nindexset}
\end{equation*}
 be and index set and define for  $(l,m) \in \mathcal{N}$ the function
\begin{equation}
    \varphi_{(l,m)} (r, \theta):= \sqrt{m+1}~ R_m^{|l|}(r)~
    \exp(il\theta), \label{varphi}
\end{equation}
 where
        \begin{equation*}
            R_m^{|l|}(r):= \sum_{j=0}^{\frac{1}{2}(m-|l|)} (-1)^j
            \frac{(m-j)!}{j! \big( \frac{m+|l|}{2}-j\big)! \big( \frac{m-|l|}{2}-j\big)!}
            r^{m-2j} \label{radialpol}
        \end{equation*}
 is the so called \textit{radial polynomial}. Finally for
 $(l, m) \in \mathcal{N}$ we define
\begin{equation}
    \psi_{(l,m)}(s, \phi):=U_m(s)\exp(il\phi),  \label{psi}
\end{equation}
 where $U_m$ denotes the $m$ths Chebyshev polynomial of the second kind. For convenience of notation we also define $\varphi_{(l,m)} \equiv 0$ and $\psi_{(l,m)}\equiv 0$ for
 $(l,m) \notin \mathcal{N}$.
 Both collections of  functions,
\begin{equation*}
    \{ \varphi_{(l,m)}: (l,m) \in \mathcal{N}\} ~~\textnormal{and}~~\{\psi_{(l,m)}:
    (l,m) \in \mathcal{N} \}
\end{equation*}
 form orthonormal bases of the spaces $\mathcal{L}^2(\mathcal{B}, \mu )$ and
 $\mathcal{L}^2(\mathcal{D}, \lambda)$ respectively. With these notations the
 SVD of $\mathcal{R}$ for some
 $\funk\in \mathcal{L}^2(\mathcal{B}, \mu )$ is given by
\begin{eqnarray}
    \mathcal{R} \funk(s, \phi) &  = & \sum_{m=0}^\infty
    \sum_{l=-m}^{ m} \frac{1}{\sqrt{m+1} }~ \psi_{(l,m)}(s, \phi) \left<\funk ,\varphi_{(l,m)}
    \right>_{\mathcal{L}^2(\mathcal{B}, \mu)}.\label{svdL2}
\end{eqnarray}
 In the literature the functions
 $\varphi_{(l,m)}(r, \theta)(m+1)^{-1/2}$
 are commonly referred to as \textit{Zernike polynomials}, which play an important role in the
 analysis of optical systems, for instance in the modelling of refraction errors, c.f.
 \cite{Ze1934} and more recently \cite{flela2011}. We refer to \cite{D1983} for more details on the
 cited SVD of the normalized Radon transform.
 Due to injectivity of the operator $\mathcal{R}$ we can immediately access its inverse $\mathcal{R}^{-1}$ pointwise defined
 for some $\R \funk \in \mathcal{R}(\mathcal{L}^2(\mathcal{B}, \mu))$, as
\begin{eqnarray}
    \funk=\mathcal{R}^{-1} \lbr \R \funk \rbr(r, \theta) &  = &
    \sum_{m=0}^\infty \sum_{l=-m}^{ m} \sqrt{m+1}~ \varphi_{(l,m)}(r, \theta)
    \left<\R\funk, \psi_{(l,m)} \right>_{\mathcal{L}^2(\mathcal{D}, \lambda)}.
    \label{inverseL2}
\end{eqnarray}
 %This inversion formula constitutes the foundation of the proposed cutoff estimator of the function $\funk$.
 The identities \eqref{svdL2}, \eqref{inverseL2} as well as $\mathcal{L}^2$-expansions in the respective spaces
 apply a priori almost everywhere. However if $g$ is sufficiently smooth they even hold uniformly. In order to specify the required regularity we define

% , but if we impose certain smoothness assumptions on $\funk$ they hold uniformly.
% To specify the regularity conditions we denote by
 \begin{equation}
    \smoothoc:=\Big  \{g \in \mathcal{L}^2(\mathcal{B}, \mu)\Big| \funk~~ \textit{continuous},
    \sum_{m=0}^\infty \sum_{l=-m}^m \Big| \lbp  \R \funk, \psi_{(l,m)}
    \rbp_{\mathcal{L}^2(\mathcal{D}, \lambda)} \Big| (m+1)^{v}< \infty, \Big\},
    \label{ellipsoidondition}
\end{equation}
the  smoothness class. We assume throughout this paper that
 the regression function $g$ in model \eqref{regressionproblem}  is an element  of $\smoothoc$  (for some $v \ge 1$).
  Controlling smoothness and thereby the complexity of the class of regression
 functions by related conditions is common in inverse problems. This is owed to their
 natural correspondence to  singular value decompositions of  operators and their
 suitability to prove minimax optimal rates (see for example
 \cite{mairruymgaart1996},  \cite{CavalierTsybakov2002}, \cite{BiHo2013} or \cite{Blanchard2018}).

\begin{proposition} \label{proposition1}
 Suppose that $\funk \in \smoothoc$ with $v \ge 1$, then the following four identities hold everywhere:
\begin{eqnarray}
    \funk  &=& \sum_{m=0}^\infty
    \sum_{l=-m}^m \varphi_{(l,m)} \lbp  \funk, \varphi_{(l,m)}
    \rbp_{\mathcal{L}^2(\mathcal{B}, \mu)}  \label{L2expansiong}    \\
     \R \funk &=& \sum_{m=0}^\infty
     \sum_{l=-m}^m \psi_{(l,m)} \left<
    \R \funk, \psi_{(l,m)} \right>_{\mathcal{L}^2(\mathcal{D}, \lambda)} \label{L2expansionRg} \\
    \mathcal{R} \funk&=&   \sum_{m=0}^\infty \sum_{l=-m}^{ m} \frac{1}{\sqrt{m+1}}
    ~ \psi_{(l,m)} \left<\funk, \varphi_{(l,m)} \right>_{\mathcal{L}^2(\mathcal{B},
    \mu)}.\label{svd} \\
    g&=& \mathcal{R}^{-1} \lbr \R \funk \rbr  =  \sum_{m=0}^\infty
    \sum_{l=-m}^{ m} \sqrt{m+1}~ \varphi_{(l,m)} \left<
    \R \funk, \psi_{(l,m)} \right>_{\mathcal{L}^2(\mathcal{D}, \lambda)}. \label{inverse}
\end{eqnarray}
 Moreover  the functions $g$ and $\R \funk$ are   $\lfloor (v-1)/2\rfloor$
 times continuously differentiable.
\end{proposition}

 The equality of $\funk$ and its $\mathcal{L}^2$-expansion is vital when
 proving uniform bounds on the distance between $\funk$ and $\est$. In one
 dimensional convolution type problems this is usually dealt with by the Dirichlet
 conditions  that directly apply to classes of smooth functions (see  \cite{nawaboppenheim1996} pp. 197-198).
% Results such as Proposition \ref{proposition1}, relating smoothness  to function classes of
% a certain geometric shape are typical in  linear inverse problems
% (see \cite{cavalier2008}.
 It should also be noted that the series condition on the function $\funk$ in
 \eqref{ellipsoidondition} implies regularity properties beyond mere smoothness.
 For instance, if $v\ge 2k+1$ it also entails periodicity of $\funk$ and its continuous
 derivatives in the angular component up to the order $k$. This property
 follows by periodicity of the basis functions in the angle and is an analogue to
 periodicity of convergent Fourier series on bounded intervals. Notice that it fits naturally to
 the scanning regime, since any function transformed from Cartesian into spherical
 coordinates will comply to periodicity with respect to the angle.

\subsection{Design} \label{sec22}

 As common in computed tomography we will assume a parallel scanning
 procedure, corresponding to a grid of design points on the detector space. Adopting our
 results to fan beam geometry, which underlies most modern scanners, is then
mathematically simple.

 We thus define a grid on the detector space $\mathcal{D}$, where for given  $p,q \in \mathbb{N}$
 each of the constituting
 rectangles has side length $1/q$ in $s$-direction and $2 \pi/p$ in $\phi$-direction.
 More formally, we define an index set
\begin{equation*}
    \mathbf{K}:=\{(k_1, k_2): 0 \le k_1 \le q-1, 0 \le k_2 \le p-1\}
\end{equation*}
 and decompose the detector space  in rectangular boxes of the  form
\begin{equation*}
    \GS:=\Big \{(s, \phi)\in \mathcal{D}:~ \frac{k_1}{q} \le s \le  \frac{k_1+1}{q},
    \frac{2 \pi k_2}{p} \le \phi \le \frac{2 \pi (k_2+1)}{p}\Big \}, \label{GSk}
\end{equation*}
 where $ \mathbf{k}=  (k_1, k_2) \in \K$. The design
 points $\{  z_{(k_1, k_2)}  ~|~(k_1, k_2) \in  \mathbf{K} \} $ are then defined as follows.
 The second coordinate of $z_{(k_1, k_2)}$ is given by
\begin{equation*}
    z_{k_2}^2:= 2 \pi \frac{k_2+\frac{1}{2}}{p}
\end{equation*}
 and  the first coordinate $z_{k_1}^1$  is determined as the solution of the equation
\begin{equation} \label{quadrature}
    \int_{k_1/q}^{(k_1+1)/q} (s-z_{k_1}^1) \sqrt{1-s^2}ds=0.
\end{equation}
 Throughout this paper we consider the inverse regression model  \eqref{regressionproblem} with these $n=pq$
 design points.
 The non-uniform design in radial direction defined by  \eqref{quadrature} is motivated by
 a midpoint
 rule to numerically  integrate over each box, with respect to the measure
 $\lambda$ in \eqref{lambda}.   For asymptotic considerations,
 we assume that $q \to \infty$ and that $p=p(q)\to \infty$ depends on $q$ as follows:

\begin{assumption} \label{assumption1}
 There exist constants $C_1$, $C_2>0$, such that $C_1 q \le p(q) \le C_2 q$ for
 all $q \in \mathbb{N}$.
\end{assumption}

 \noindent
 Denoting the number of rows and columns in the
 grid of design points by $q$ and $p$ respectively is common in the literature and
 numerical programming. Notice that our Assumption \ref{assumption1}  leaves room
 for the resolution optimal choice $2 \pi q \approx  p$ (see \cite{NaW2001},
  p. 74).
 Sometimes we will use the notation $n \to \infty$, actually meaning
 that according to Assumption \ref{assumption1} $q$ and thereby $p$ and $n$ diverge.  Note also
that  the   index set $\K$ depends on the sample size $n$
in model \eqref{regressionproblem}. Thus  formally we consider a triangular array of independent, identically distributed
and
centred  random variables $(\varepsilon_\k)_{\k \in \K}$, but we do not
reflect this dependence  on $n$  in our notation.

\subsection{The spectral cutoff estimator} \label{spec_cut_est}

 Motivated by the representation \eqref{inverse} we now define the cutoff estimator
 $\est$ for the function $\funk$ in model \eqref{regressionproblem} by
\begin{equation}
    \est (r, \theta)=  \sum_{m = 0}^{t_n} \sum_{l=-m}^{m} \sqrt{m+1}  ~
    \varphi_{(l,m)}(r, \theta) ~ \rpN. \label{estimator}
\end{equation}
 Here 
\begin{equation}
    \rpN:= \sum_{\mathbf{k} \in \mathbf{K}}  \weightk\overline{\psi_{(l,m)}}(\desk)~
    Y_{\mathbf{k}}\label{R_N}
\end{equation}
 is an estimator of the inner product
\begin{equation}
    \rp:=\left<
    \R \funk, \psi_{(l,m)} \right>_{\mathcal{L}^2(\mathcal{D}, \lambda)}
     \label{R}
\end{equation}
and   $\weightk:=\lambda(\GSk)$ denotes  the Lebesgue measure of the cell $\GSk$.
Comparing  \eqref{inverse} to our estimator in \eqref{estimator}, we observe that the inner products have been replaced by the estimates  \eqref{R_N}. Furthermore the series has been truncated at  $t_n \in \mathbb{N}$, which represents the application of a regularized inverse. 
% Note that  the inner products in the expansion \eqref{inverse} are replaced by the
 %estimates \eqref{R_N} and the parameter  $t_n \in \mathbb{N}$  is used for a truncation representing the
% application of a regularized inverse.
 In the literature it is common to refer to either $t_n$ or $t_n^{-1}$ as
 \textit{bandwidth}, since it is used to balance between bias and variance like the bandwidth in
 kernel density estimation (see \cite{cavalier2008}).

 The choice of a bandwidth is a
 non-trivial problem. An optimal bandwidth with respect to some criterion such
 as the integrated mean squared error will depend on the unknown regression
 function $\funk$. Several data driven selection criteria for the choice of $t_n$ have
 been proposed and examined in the literature. We refer to the monograph of \cite{Vo2002},
 where multiple techniques are gathered. More closely related to our case is the risk hull method
 by \cite{cavalier2006} in the white noise model and the smooth bootstrap examined
 by \cite{BiChDe2018} in a different context.

\begin{remark}
 It should be noticed that in practice a smooth dampening of high frequencies usually
 shows a better performance  than the strict spectral cutoff. We can accommodate this  by
 introducing a smooth version of the estimator $\est$ in \eqref{estimator}. For
 this purpose let $\Lambda: \mathbb{R} \to [0,1]$ denote a function with
 compact support and define
\begin{equation} \label{altest}
    \est_{\Lambda} (r, \theta)=  \sum_{m=0}^\infty ~ \Lambda(m t_n^{-1})\sum_{l=-m}^{m}
    \sqrt{m+1} \varphi_{(l,m)}(r, \theta) ~ \rpN,
\end{equation}
 as an alternative estimator of $\funk$. Note that the estimate $\est$ in
 \eqref{estimator} is obtained for $\Lambda(x)=\mathbbm{1}_{[0,1]}(x)$.
 All results presented in this paper  remain valid for the estimator \eqref{altest}.
 However, for sake of brevity and a transparent presentation the subsequent
 discussion is restricted to the spectral cutoff  estimator in \eqref{estimator}.
\end{remark}

\section{The empirical process of residuals}
\label{sec3}
In this section we investigate the asymptotic properties of the empirical residual process
    $$\sqrt{n}(\hF(t)-\F(t)) := \sqrt{n}\sum_{ \mathbf{k} \in \mathbf{K}}
    \weight \big(\mathbbm{1}\left\{ \resk \le t\right\}-\F(t)\big),~ t \in \RR, $$
 where $\F$ denotes the residual distribution function and
  \begin{equation}
    \resk := Y_\mathbf{k}-\mathcal{R}\est(\desk), ~ \mathbf{k} \in \mathbf{K} \label{estimatedresidual}
\end{equation}
 the  $\mathbf{k}$th residual obtained from  the estimate
 $\hat g$.
 The weights $\weightk$ are defined in Section \ref{spec_cut_est}.
 We begin by showing a uniform convergence result for $\est$. For this
 purpose we derive uniform approximation rates  for bias and variance and subsequently
 balance these two, to get optimal results. The proofs of the following
 results are complicated and therefore deferred to the Appendix.

\begin{lemma} \label{lemma1}
 Suppose that Assumption \ref{assumption1} holds and that $\funk \in \smoothoc$
  for some $v\ge 5$. Then  the  estimator $\hat g$ in  \eqref{estimator} satisfies
\begin{equation*}
    \big\| \mathbb{E}\est(z)-\funk(z)\big\|_\infty =
    \mathcal{O}\lb t_n^{-(v-1)} +  t_n^{8}n^{-1}  \rb,
\end{equation*}
where $\| g\|_\infty:=\sup_{z \in \mathcal{B}} |g(z)|$.
\end{lemma}

Next we derive a uniform bound for the random error of  the  estimator $\hat g$. %The strong moment
 %assumptions result from the triangular form of the error in the regression problem
 \eqref{regressionproblem}.

\begin{lemma} \label{lemma2}
 Suppose that Assumption \ref{assumption1} holds and that
 $\mathbb{E}|\varepsilon|^ \kappa<\infty$ for some $\kappa>3$. Additionally let
 the sequence $(t_n)_{n \in \mathbb{N}}$ satisfy $t_n n^{-1/2} =
 \mathcal{O}(1)$. Then the  estimator $\hat g$ in  \eqref{estimator} satisfies
\begin{equation*}
    \big\|\est(z)-\mathbb{E}\est(z)\big\|_\infty= \mathcal{O}
    \big( t_n^{4}     \log(n)^{1/2}n^{-1/2} \big) ~~~a.s.
\end{equation*}
\end{lemma}

 Balancing the two upper bounds for the deterministic and random part yields an optimal choice of
 the bandwidth. More precisely for $v\ge 5$ the choice
\begin{equation}
    t_n := \Theta \left( \left(\log(n)^{-1} n \right )^\frac{1}{2(v+3)} \right )
    \label{truncationparameter}
\end{equation}

 \noindent balances the upper bound from Lemma \ref{lemma2} with the leading term
 $\mathcal{O}(t_n^{-(v-1)})$ of the bias from Lemma \ref{lemma1}. Combining these
 results yields the first part of the following   theorem.

\begin{theorem} \label{theorem1}
 Let Assumption \ref{assumption1} hold, suppose that $\funk \in \smoothoc$ for some
 $v\ge 5$ and that $\mathbb{E}|\varepsilon|^\kappa<\infty$ for some $\kappa>3$.
 Additionally let $t_n$ be chosen as in \eqref{truncationparameter}. Then
\begin{equation}
    \| \funk(z)-\est(z)\|_\infty =
    \mathcal{O}\left(\left(\log(n)n^{-1}\right)^{\frac{v-1}{2(v+3)}} \right)
\end{equation}
 and for all $\tau\le v$
\begin{equation}
    \sum_{m=0}^\infty \sum_{l=-m}^m m^\tau \big| \lbp \mathcal{R}
    \lbr\funk-\est\rbr,  \psi_{(l,m)} \rbp_{\mathcal{L}^2(\mathcal{D}, \lambda)} \big|
    =\mathcal{O}\lb    n^{\frac{v-\tau}{2(v+3)}}\rb~~~a.s. \label{ellipsedecay}
\end{equation}
\end{theorem}

 By the same techniques uniform bounds can be deduced for the derivatives of our
 estimators.

\begin{corollary} \label{corollary1}
 Let the assumptions of Theorem \ref{theorem1}  hold, let $t_n$ be of order $o(n^{1/4})$ and
 suppose $v \ge 2k+1$ for some  $k \in \mathbb{N}_0$. Additionally let $\alpha, \beta \in
 \mathbb{N}_0$, such that $\alpha+\beta=k$. Then

\begin{equation}
    \left\| \frac{\partial^\alpha}{\partial r^\alpha}\frac{\partial^\beta}
    {\partial \theta^\beta}\funk- \frac{\partial^\alpha}{\partial r^\alpha}
    \frac{\partial^\beta}{\partial \theta^\beta} \est \right\|_\infty = \mathcal{O}\lb\lb
    \log(n)^{1/2}n^{-1/2} t_n^{2k+4}+t_n^{v-(2k+1)} \rb\rb ~~~a.s.
\end{equation}
\end{corollary}

In order to prove the weak convergence of the process
 $\sqrt{n}(\estcdf-\cdf)$ we consider the {\em bracketing metric entropy}
of the subclass
\begin{equation}
    \smoothel2:=\Big\{\funk:\mathcal{B} \to \mathbb{R}\big| \funk~~
     \textit{continuous},
    \| \funk\|_\infty \le 1,  \sum_{m=0}^\infty
     \sum_{l=-m}^m \la \rp \ra (m+1)^{\smoothindex}\le 1, \Big\},
     \label{octahedralcondition2}
\end{equation}
 for some $\tau>0$. Theorem \ref{theorem1} implies that for all $\tau<v$ the difference $\hat{g}-g$ eventually lies in $\smoothel2$. As we know from Proposition \ref{proposition1} the condition 
 $$
 \sum_{m =0}^\infty \sum_{l=-m}^m (m+1)^\tau \big|\left<\R h, \psi_{(l,m)} \right>_{\mathcal{L}^2(\mathcal{D}, \lambda)}\big|\le 1
 $$
 entails that a function $h \in \mathcal{L}^2(\mathcal{B}, \mu)$ is smooth to
 a degree determined by $\tau$.  This implies that a  finite-dimensional representation
 can be used as an adequate approximation of $h$, in our case a
 truncated
 $\mathcal{L}^2$-expansion. We employ these considerations to derive the
 following result about the complexity of the class $\smoothel2$, which is of
 independent interest and is proven in Appendix B (see section \ref{ApB4}).

\begin{proposition} \label{proposition5}
 Let $\smoothindex>3$, then for any $t \in (3, \smoothindex)$ and sufficiently small
 $\epsilon>0$
\begin{equation}
    \log \lb N_{[]}(\epsilon, \smoothel2, \|\cdot\|_\infty) \rb \le C \lb
    \frac{1}{\epsilon}\rb^{\frac{2}{\smoothindex-t}}.
\end{equation}
 $N_{[]}(\epsilon, \smoothel2, \|\cdot\|_\infty) $ denotes the minimal number of
 $\epsilon$-brackets with respect to  $\| \cdot \|_\infty$ needed to cover the
 smoothness class $\smoothel2$.
\end{proposition}

 For the next step recall the definition of the estimated residuals $\resestk$ in
 \eqref{estimatedresidual}, as well as the estimate for the residual distribution 
 function $\estcdf$ in \eqref{estimateddistributionfunction}. In order to prove a uniform CLT for
 $\sqrt{n}(\estcdf-\cdf)$ we disentangle the dependencies of the terms in $\estcdf$ in the
 next result.

\begin{theorem} \label{theorem2}
 Assume that $\funk \in \smoothoc$ for some $v >5$,
 $\mathbb{E}|\varepsilon|^\kappa<\infty$ for some $\kappa>3$, that $\cdf$ admits a
  H\"older continuous density $\dens$ with exponent $\zeta>4/(v-1)$ and that Assumption
  \ref{assumption1} holds. If the bandwidth $t_n$ satisfies \eqref{truncationparameter}, then
\begin{equation}
    \sup_{t \in \RR} \la \sum_{\mathbf{k}\in \K} \weight \left[ \mathbbm{1}\{
    \resestk \le t\} -\mathbbm{1}\{\varepsilon_\mathbf{k }\le t\} - \varepsilon_\k
    \dens(t) \right] \ra =o_P \lb n^{-1/2}\rb.
\end{equation}
\end{theorem}

\begin{corollary}
 \label{corollary2}
 Under the assumptions of Theorem \ref{theorem2}, the process
\begin{eqnarray*}
    & & \Big\{ \sum_{\k \in \K} n^{1/2} \weight \left\{ \mathbbm{1}
    \{\resestk \le t\} - \cdf(t) \right\} \Big\}_{t \in \mathbb{R}}
\end{eqnarray*}
 converges weakly to a mean zero Gaussian process $G$ with covariance function
\begin{eqnarray*}
    \Sigma(t,\tilde{t}) & := &   \frac{8 \pi^2}{3}\Big(\cdf(\min(t,\tilde{t}))-\cdf(t)
    \cdf(\tilde{t})+\dens(t)\mathbb{E}\left[\varepsilon \mathbbm{1}\{\varepsilon \le
    \tilde{t}\}\right] \\[1ex]
    & + & \dens(\tilde{t})\mathbb{E}\left[\varepsilon \mathbbm{1}\{\varepsilon \le t\}
    \right] + \sigma^2 \dens(t)\dens(\tilde{t}) \Big),~~~t, \tilde{t} \in \RR.
\end{eqnarray*}
\end{corollary}

\medskip
\medskip

{\bf Acknowledgements}
This work has been supported in part by the Collaborative Research
Center ``Statistical modeling of nonlinear dynamic processes'' (SFB
823, Project C1) of the German Research Foundation (DFG) and
 the Bundesministerium f\"ur Bildung und Forschung through
the project ``MED4D: Dynamic medical imaging: Modeling and analysis of
medical data for improved diagnosis, supervision and drug
development''.

\appendix

\section{Proofs and technical details} \label{proofsanddetails}

\numberwithin{equation}{section}
Throughout our calculations $C$ will denote a positive constant, which may differ from
 line to line. The dependence of  $C$ on other parameters will be highlighted in the
 specific context.

\subsection{Proof of Lemma \ref{lemma1}}
 We begin with an auxiliary result which provides  an approximation
 rate for Lemma \ref{lemma1} in expectation of $\rpN$ for $\rp$ and is
 proven in Appendix B (see section \ref{ApB3}).

 \begin{proposition} \label{proposition4}
 Suppose that $\funk \in \smoothoc$ for $v\ge 5$ and that Assumption \ref{assumption1}
 holds. Then for all $(l,m) \in \mathcal{N}$ it follows that
\begin{equation}
    |\mathbb{E}\rpN-\rp| \le C m^5 n^{-1}
\end{equation}
 where $C>0$ is some constant depending on $\funk$ and $C_1$ (the constant from
 Assumption \ref{assumption1}).
\end{proposition}

 We are now in a position to derive the decay rate of the bias
  postulated in Lemma \ref{lemma1}. The decay rate naturally splits up into two parts. One accounts for
 the average approximation error of Radon coefficients with index $m$ smaller than $t_n$ and the other for the error due to frequency limitation of the estimator.

The singular
 value decomposition of the normalized Radon transform in
 \eqref{svd} and the definition of our estimator (in \eqref{estimator}) yield
\begin{eqnarray}
    \label{lemma1bias}  \| \mathbb{E}\est-\funk \|_\infty  = \left\| \sum_{m = 0 }^{t_n}
    \sum_{l=-m}^m \sqrt{m+1}\varphi_{(l,m)} \lb \mathbb{E}  \rpN -  \rp \rb \right\|_\infty
    \le A_1 + A_2, \nonumber
\end{eqnarray}
where the terms $A_1$ and $A_2$ are given by
\begin{eqnarray}
    & &A_1 := \sum_{m = 0 }^{t_n} \sum_{l=-m}^m \sqrt{m+1} ~\left\|\varphi_{(l,m)}
    \right\|_\infty~ \la\mathbb{E}\rpN-\rp \ra   \nonumber \\[1ex]
    & & A_2 := \sum_{m > t_n } \sum_{l=-m}^m \sqrt{m+1} ~\|\varphi_{(l,m)}\|_\infty~\la \rp
    \ra. \nonumber
\end{eqnarray}
 For the term $A_1$ it follows that
\begin{eqnarray*}
    A_1 & \le & \sum_{m = 0 }^{t_n} \sum_{l=-m}^m (m+1) \la\mathbb{E}\rpN-\rp \ra \le
    \sum_{m=0}^{t_n} \sum_{l=-m}^m  (m+1)C m^5 n^{-1}  = \mathcal{O}\lb t_n^8 n^{-1} \rb,
\end{eqnarray*}
 where we have used that Proposition \ref{ApB2} in Appendix B implies the estimate
\begin{equation}
    \| \varphi_{(l,m)} \|_\infty \le \sqrt{m+1} \label{varphibound}
\end{equation}     in the first and the approximation result from Lemma
 \ref{proposition4} in the second inequality. Similarly we have

\begin{eqnarray*}
    A_2 & \le &     \sum_{m>t_n} \sum_{l=-m}^m (m+1) |\rp| \le \sum_{m>t_n}
    \sum_{l=-m}^m (m+1)^v t_n^{1-v} |\rp| \\[1ex]
    &\le & t_n^{1-v}     \sum_{m = 0 }^\infty\sum_{l=-m}^m (m+1)^v  |\rp|  =
    \mathcal{O}\lb t_n^{1-v}\rb.
\end{eqnarray*}
 In the last step we have used that $\funk$ complies to the smoothness  condition of $\smoothoc$
 (see \eqref{ellipsoidondition}) and thus the series converges. \qed

\subsection{Proof of Lemma \ref{lemma2}}
 We first rewrite $\est-\mathbb{E}\est$ employing \eqref{R_N} and \eqref{varphibound}
\begin{eqnarray} \nonumber
    \label{biasterm} \Big\| \est(z)-\mathbb{E}\est(z)\Big\|_\infty &=&\Big\| \sum_{m=0}^{t_n}
    \sum_{l=-m}^m \sqrt{m+1} \varphi_{(l,m)} \lb \rpN-\mathbb{E}\rpN\rb \Big\|_\infty
    \nonumber\\[1ex]
    & = & \Big\| \sum_{m=0}^{t_n}  \sum_{l=-m}^m \sqrt{m+1} \varphi_{(l,m)} \lb \sum_{\k
    \in \K} \overline{\psi_{(l,m)}}(\desk) \weight \varepsilon_\k \rb \Big\|_\infty \nonumber \\[1ex]
    & \le & \sum_{m=0}^{t_n}  \sum_{l=-m}^m (m+1)~ \Big| \sum_{\k \in \K}
    \overline{\psi_{(l,m)}}(\desk) \weight \varepsilon_\mathbf{k} \Big| \nonumber \\[1ex]
    & = & \sum_{m=0}^{t_n}  \sum_{l=-m}^m (m+1)^2~ \Big| \sum_{\k \in \K}
    \overline{\psi_{(l,m)}}(\desk) ~(m+1)^{-1} \weight \varepsilon_\mathbf{k} \Big| \nonumber \\[1ex]
    & \le & \sum_{m =0}^{t_n}  2(m+1)^3   \max_{\substack{(l,m) \in \mathcal{N} \\ m \le
    t_n}} \Big| \sum_{\k \in \K} \overline{\psi_{(l,m)}}(\desk)~(m+1)^{-1} \weight \varepsilon_\k
    \Big| \nonumber\\[1ex]
    & \le & C t_n^4 \max_{\substack{(l,m) \in \mathcal{N} \\ m \le t_n}} \Big| \sum_{\k \in
    \K} \overline{\psi_{(l,m)}}(\desk)~(m+1)^{-1} \weight \varepsilon_\k \Big| \nonumber
\end{eqnarray}
 We proceed deriving an upper bound for the maximum. For this purpose we introduce a truncation
 parameter $\trunc:=n^{1/2} \log(n)^{-1/2}$ and define the truncated error
\begin{equation}
    \truncres := \mathbbm{1}\left\{ \la \varepsilon_\k \ra \le \trunc
    \right\} \varepsilon_\k. \label{epsnot}
\end{equation}
 We will now show that all of the errors $\varepsilon_\k$
 with $\k \in \K$ eventually equal their truncated versions $\truncres$ almost surely.
 Via Markov's inequality we conclude that
\begin{equation*}
    \mathbb{P}\lb \la \varepsilon_\k \ra > \trunc\rb \le \mathbb{E}\lbr \la \varepsilon
    \ra ^\kappa \rbr \trunc^{-\kappa}
\end{equation*}
 and therewith it follows that
\begin{equation*}
    \sum_{n=pq} \mathbb{P}\lb \exists
    \k \in \K: \truncres \neq \varepsilon_\k  \rb= \sum_{n=pq} \mathbb{P}\lb \exists
    \k \in \K:  \la \varepsilon_\k \ra > \trunc \rb \le \sum_{n=pq} n \trunc^{-\kappa}
    \mathbb{E}\lbr \la \varepsilon \ra^\kappa \rbr.
\end{equation*}
 Recalling that $n=pq$ and that there exists some $C_2>0$ such that $p \le C_2 q$ by
 Assumption \ref{assumption1}, we derive
\begin{equation*}
    C \sum_{n=pq} n \trunc^{-\kappa} =C \sum_{n=pq} n^{1-\kappa/2}  \log(n)^{\kappa/2} \le C
    \sum_{q \ge 1} q^{2-\kappa} \log(C_2 q^2)^{\kappa/2}<\infty.
\end{equation*}
 Summability is entailed by $2-\kappa<-1$.
 The Borel-Cantelli Lemma implies that almost surely eventually all measurement errors and their truncated
 versions are equal. Thus we can confine ourselves to the maximum
\begin{eqnarray}
    & &\max_{\substack{(l,m) \in \mathcal{N} \\ m \le t_n}} \Big| \sum_{\k \in \K}
    \overline{\psi_{(l,m)}}(\desk)~(m+1)^{-1} \weight \truncres \Big| \le B_1 +B_2, \label{centeredsum}
\end{eqnarray}
where $B_1$ and $B_2$ are defined by
\begin{eqnarray*}
    & &B_1:= \max_{\substack{(l,m) \in \mathcal{N} \\ m \le t_n}} \Big| \sum_{\k \in \K}
    \overline{\psi_{(l,m)}}(\desk)~(m+1)^{-1} \weight \lbr\truncres-\mathbb{E}\truncres\rbr \Big| \\[1ex]
    & & B_2:=    \sum_{\k \in \K} \la \overline{\psi_{(l,m)}}(\desk) \ra ~(m+1)^{-1} \weight \la  \mathbb{E}
\truncres \ra .
\end{eqnarray*}
 Using the inequality
\begin{equation}
    \label{psibound} \| \psi_{(l,m)} \|_\infty \le m+1,
 \end{equation} which is a consequence of Proposition \ref{proposition3}, it follows that
\begin{equation*}
    B_2     \le  \la  \mathbb{E} \varepsilon^{\trunc} \ra  \sum_{\k \in \K}
    \weight = \mathcal{O}
    (n^{-1/2}),
\end{equation*}
 wherewe exploit the decay rate
 $\la  \mathbb{E} \varepsilon^{\trunc} \ra  = \mathcal{O}(n^{-1/2})$ in the last estimate .
 For the proof of this fact we recall the notation \eqref{epsnot} and note that the condition
 $\mathbb{E} \varepsilon=0$ implies
\begin{equation*}
    |\mathbb{E}\varepsilon^{\trunc}| =\la \mathbb{E} \lbr \varepsilon \mathbbm{1}\{|\varepsilon| >\trunc\}  \rbr \ra \le
    \int_{\trunc }^\infty \mathbb{P} \lb \la \varepsilon \ra  >s\rb ds \le \mathbb{E}
    \lbr \la \varepsilon \ra^\kappa \rbr (\kappa-1)^{-1} \trunc^{1-\kappa}=
    \mathcal{O}\lb n^{-1/2} \rb.
\end{equation*}
 For the term $B_1$ we note that for a fixed constant $C^\star$
\begin{eqnarray}
    & & \mathbb{P}\lb \la B_1 \ra > \log(n)^{1/2}n^{-1/2} C^\star \rb \label{unionboundeq} \\[1ex]
    & \le & t_n^2 \max_{\substack{(l,m) \in \mathcal{N} \\ m \le t_n}}\mathbb{P}\lb  \la
    \sum_{\k \in \K} \overline{\psi_{(l,m)}}(\desk)~(m+1)^{-1} \weight \lbr\truncres-\mathbb{E}
    \truncres\rbr \ra > \log(n)^{1/2}n^{-1/2} C^\star \rb. \nonumber
\end{eqnarray}
 Due to truncation $|\varepsilon^{\trunc}_\k-\mathbb{E}\varepsilon^{\trunc}_\k|$ is
 bounded by $2 \trunc$ and its variance by $ \sigma^2$. Furthermore the weights are uniformly of order $\mathcal{O}(n^{-1})$, since
\begin{equation*}
    \max_{\k \in \K} \weight= 2\pi^{-2}  \max_{\k \in \K} \int_{2 \pi k_2/q}^{2 \pi
    (k_2+1)/q} \int_{k_1/p}^{(k_1+1)/p} \sqrt{1-s^2}ds~ d\phi
    \le 4 (\pi pq)^{-1}=4 (\pi n)^{-1}.
\end{equation*}
 Consequently the Bernstein inequality yields for the right
 side of \eqref{unionboundeq} the upper bound
\begin{eqnarray*}
    & &t_n^2 \exp\Big(-\frac{ C C^\star \log(n)/n}{ n^{-1}+ \log(n)^{\frac{1}{2}}\trunc/
    n^{\frac{3}{2}} } \Big) \le  t_n^2 \exp\lb -C C^\star \log(n) \rb \le
    t_n^2 n^{-C C^\star},
\end{eqnarray*}
 which is summable for sufficiently large $C^\star$. The Borel-Cantelli Lemma therefore implies
 that
\begin{equation*}
    \max_{\substack{(l,m) \in \mathcal{N} \\ m \le t_n}} \Big| \sum_{\k \in \K}
    \overline{\psi_{(l,m)}}(\desk)~(m+1)^{-1} \weight \lbr\truncres-\mathbb{E}\truncres\rbr \Big| =
    \mathcal{O}\lb \log(n)^{1/2}n^{-1/2} \rb ~~~a.s.
\end{equation*}
 Combining these estimates we see, that the left side of \eqref{centeredsum} is almost
 surely of order 
 \\ $\mathcal{O}(\log(n)^{1/2}n^{-1/2})$. Consequently the right side
 of \eqref{biasterm} is of order $\mathcal{O}(t_n^4 \log(n)^{1/2}n^{-1/2})$ almost
 surely, which proves the assertion. \qed

\subsection{Proof of Theorem \ref{theorem1}}\label{secA3}

 Combining Lemma \ref{lemma1} and Lemma \ref{lemma2} yields
  the first part of Theorem \ref{theorem1}, when the truncation parameter
 is chosen as in \eqref{truncationparameter}. For the proof of the second
 property  we note the identity

\begin{eqnarray*}
    & & \lbp \mathcal{R}\lbr\funk-\est\rbr, \psi_{(l,m)}  \rbp_{\mathcal{L}^2
    (\mathcal{D}, \lambda)}=  \rp- \rpN \mathbbm{1}\{m \le t_n\},
\end{eqnarray*}
which gives for the left hand side of \eqref{ellipsedecay}
\begin{eqnarray}
    & & \sum_{m=0}^\infty \sum_{l=-m}^m m^\tau \la \lbp
    \mathcal{R}\lbr\funk-\est\rbr , \psi_{(l,m)} \rbp_{\mathcal{L}^2(\mathcal{D}, \lambda)} \ra
    \label{thm1eq1} = \sum_{m=0}^\infty \sum_{l=-m}^m m^\tau
    \la \rp- \rpN \mathbbm{1}\{m \le t_n\} \ra \\[1ex]
    & \le & D_1+D_2+D_3. \nonumber
\end{eqnarray}
The terms $D_1$, $D_2$ and $D_3$ are defined as follows:
\begin{eqnarray*}
 D_1& :=& \sum_{m =0}^{t_n} \sum_{l=-m}^m m^\tau \la \rp -\mathbb{E}\rpN \ra \\[1ex]
 D_2&:=&\sum_{m =0}^{t_n} \sum_{l=-m}^m m^\tau \la \rpN -\mathbb{E}\rpN
    \ra \\[1ex]
    D_3& := &\sum_{m >t_n}  \sum_{l=-m}^m m^\tau \la \rp\ra.
\end{eqnarray*}
By Proposition \ref{proposition4}  we receive the
 upper bound
\begin{equation*}
    D_1    \le \sum_{m =0}    ^{t_n} \sum_{l=-m}^m C m^{\tau+5}n^{-1}
    =\mathcal{O}\lb t_n^{\tau+7}n^{-1}\rb.
\end{equation*}
  For the second sum on right of \eqref{thm1eq1}  we use the estimate
\begin{eqnarray*}
    & & D_2 =     \sum_{m =0}^{t_n} \sum_{l=-m}^m m^{\tau}(m+1) \lbr (m+1)^{-1} \la \rpN -\mathbb{E}\rpN
    \ra \rbr \\[1ex]
    & \le & C  t_n^{\tau+3} \max_{\substack{(l,m) \in \mathcal{N} \\ m \le t_n}}
    \left\{ (m+1)^{-1} \la \rpN -\mathbb{E}\rpN \ra \right\} = \mathcal{O}\lb n^{-1/2}
    \log(n)^{1/2} t_n^{\tau+3} \rb~~a.s.
\end{eqnarray*}
 In the last equality we have used the following bound established in the proof of
 Lemma \ref{lemma2}:
\begin{eqnarray*}
    & & \max_{\substack{(l,m) \in \mathcal{N} \\ m \le t_n}}
    \left\{ (m+1)^{-1} \la \rpN -\mathbb{E}\rpN \ra \right\} \\[1ex]
    & = &\max_{\substack{(l,m) \in \mathcal{N} \\ m \le t_n}} \Big| (m+1)^{-1}\sum_{\k \in \K}
    \overline{\psi_{(l,m)}}(\desk) \weight \varepsilon_\k \Big|
    = \mathcal{O}\lb n^{-1/2}
    \log(n)^{1/2}\rb~~~a.s.
\end{eqnarray*}
The third term in \eqref{thm1eq1} can be bounded by
\begin{equation*}
    D_3 \le   \sum_{m >t_n}
    \sum_{l=-m}^m (m t_n)^{v-\tau} m^\tau  \la \rp\ra= t_n^{v-\tau} \sum_{m >t_n}
    \sum_{l=-m}^m m^v  \la \rp\ra.
\end{equation*}
Due to the smoothness condition in \eqref{ellipsoidondition} the double sum is finite.
 Since $t_n \to \infty$ it follows that the series converges to $0$ for $n \to \infty$.
 Consequently
\begin{equation*}
    t_n^{v-\tau} \sum_{m >t_n}  \sum_{l=-m}^m m^v  \la \rp\ra =
    o\lb t_n^{v-\tau} \rb
\end{equation*}
 and the definition of $t_n$ in \eqref{truncationparameter}
 yield the desired result. \qed

\subsection{Proof of Theorem \ref{theorem2} \label{secA4}}

 Proposition \ref{proposition5} is used to verify an equicontinuity
 argument, which is the central building block in the proof of Theorem \ref{theorem2}.
 For this purpose we define the $L^2_n$-bracketing number as follows:

\begin{definition} \label{L_2^Ndef}
 Let $Z_{1,n},...,Z_{n,n}$  be stochastic processes, indexed in  $\mathcal{F}$, and
 $\epsilon>0$.  The $L_2^n$-bracketing number of $\mathcal{F}$, denoted by
 $N_{[]}^{L_2^n}(\epsilon, \mathcal{F})$, is the minimal number $N_\epsilon$ of sets
 $\mathcal{F}_{\epsilon, j}^n$ in a partition of
 $\mathcal{F}= \bigcup_{j=1}^{N_\epsilon} \mathcal{F}_{\epsilon, j}^n$  such that for
 each $j$
\begin{equation}
    \sum_{i=1}^n \mathbb{E} \Big[ \sup_{f,\tilde{f} \in \mathcal{F}_{\epsilon, j}^n}
    \lb Z_{i,n}(f) -Z_{i,n}(\tilde{f})\rb^2 \Big] \le \epsilon^2. \label{L_2^N}
\end{equation}
\end{definition}

\begin{lemma} \label{lemma3}
Define
\begin{eqnarray}
    M_n(t) :=   \bigg|  \sum_{\k \in \K} \weight \Big\{ \mathbbm{1} \{ \resestk \le t\} -
    \mathbbm{1}\{\varepsilon_\mathbf{k} \le t\} + \cdf(t)   - \mathbb{P}\lb
    \hat{\varepsilon}_\mathbf{k}\le t \rb \Big\} \bigg|.\label{MNdef}
\end{eqnarray}
 Then, under the assumptions of Theorem \ref{theorem2} it follows that
 $\sup_{t \in \mathbb{R}}|M_N(t)|=o_P(n^{-1/2})$.
 \end{lemma}

\begin{proof}[Proof of Lemma A.1]
 \noindent Using the definition of the estimated residuals in \eqref{estimatedresidual}
 we have
\begin{eqnarray*}
    & &  \mathbbm{1} \{ \resestk \le t\} - \mathbbm{1}\{\varepsilon_\mathbf{k}
    \le t\} + \cdf(t)  -  \mathbb{P}\lb \resestk \le t \rb \\[1ex] & = &
    \mathbbm{1} \left\{ \varepsilon_\mathbf{k} \le t +  \mathcal{R}[\est-\funk]
    (\desk)  \right\} - \mathbbm{1} \{ \varepsilon_\mathbf{k} \le t\}+
    \mathbb{P}\lb \varepsilon \le t \rb  - \mathbb{P}\lb \varepsilon_\mathbf{k}
    \le t +  \mathcal{R}[\est-\funk](\desk) \rb.
\end{eqnarray*}
 As we have seen in Theorem \ref{theorem1}, the random function
 $d_n:=\R[\funk-\est]$ is eventually included in the  smoothness
 class $\R(\smoothel2)$ for every $\tau<v$. Since $v>5$ by assumption,
 we can also choose a $\tau>5$.
 Since $d_n$ is a complicated object, depending on all residuals, we
 replace it by general functions in $\R(\smoothel2)$ and
  prove a uniform result over $\smoothel2$. We thus define
 the stochastic processes
\begin{equation*}
    Z_{n, \mathbf{k}}(t,d):= n^{1/2} \weight  \lb \mathbbm{1} \left\{
    \varepsilon_\mathbf{k} \le t +  d(z_\mathbf{k})  \right\} -
    \mathbbm{1} \{ \varepsilon_\mathbf{k} \le t\} \rb,
\end{equation*}
 indexed in the space $\mathcal{F}:= \mathbb{R}\times \R(\smoothel2)$,
 equipped with  the semi metric
\begin{equation}
    \rho((t,d),(\tilde{t}, \tilde{d})):= \max \Big\{ \sup_{x \in [-1, 1]}\big\{|\cdf(t+x)
    -\cdf(\tilde{t}+x)|, \|d-\tilde{d}\|_\infty \big\}\Big\}. \label{semimetric}
\end{equation}
 Notice that for $\rho$ to be a semimetric the error density $\dens$ must have support
 $\mathbb{R}$, which is assumed at this point for the sake of simplicity. Furthermore
 recall the uniform order of the product
 $n^{1/2}\weight  = \mathcal{O}(n^{-1/2})$. To prove equicontinuity we have to
 show that for every sequence $\delta_n \downarrow 0$ and every
 $\epsilon>0$
\begin{equation}
    \mathbb{P}\Bigg( \sup_{\substack{ (t,d), (\tilde{t}, \tilde{d}) \in \mathcal{F} \\
    \rho((t,d), (\tilde{t}, \tilde{d})  )<\delta_n} } \la \sum_{\k \in \K}\lb
    Z_{n, \mathbf{k}}(t, d) - \mathbb{E} Z_{n, \mathbf{k}}(t, d)-Z_{n, \mathbf{k}}
    (\tilde{t}, \tilde{d}) + \mathbb{E} Z_{n, \mathbf{k}}(\tilde{t}, \tilde{d}) \Bigg)
    \ra > \epsilon\rb \to 0. \label{equicontinuity}
\end{equation}
 If \eqref{equicontinuity} holds, then the assertion of Lemma \ref{lemma3} can
 be shown as follows: Firstly note that we can derive a lower
 bound for the probability on the left hand side of \eqref{equicontinuity} by
\begin{eqnarray}
    & & \mathbb{P}\Big(\Big\{ \sup_{t \in \mathbb{R}} \big|\sum_{\k \in \K}
    Z_{n, \mathbf{k}}(t, d_n) - \mathbb{E} Z_{n, \mathbf{k}}(t, d_n)-
    Z_{n, \mathbf{k}}(t, 0) + \mathbb{E} Z_{n, \mathbf{k}}(t, 0)\big|>\epsilon
     \Big\} \label{lemmathm2est1}\\[1ex]
    & &~~ ~~~~\bigcap \Big\{ d_n \in \R(\smoothel2)\Big\}\bigcap
    \Big\{\sup_{t \in \mathbb{R}} \rho\Big( (t,0), (t, d_n) \Big) < \delta_n
    \Big\}\Big) \nonumber \\[1.5ex]
    &=  &   \mathbb{P}\Big( \Big\{\sup_{t \in \mathbb{R}} \la n^{1/2} M_n(t)
    \ra >\epsilon \Big\} \bigcap\Big\{ d_n\in \R(\smoothel2) \Big\}
    \bigcap \Big\{ \sup_{t \in \mathbb{R}}\rho \big( (t,0), (t, d_n) \big)
    < \delta_n \Big\} \Big). \nonumber
\end{eqnarray}
 By the second part of Theorem \ref{theorem1} we know that
    $$ \mathbb{P}\lb  d_n \in \R(\smoothel2) \rb \to 1,$$
 for $\tau < v$. Furthermore we notice that $\| \R d \|_\infty \le \| d \|_\infty$ for all continuous function $d$, which follows
 immediately from the definition of the Radon transform in
 \eqref{radontransform}. Combining this, with the upper bound
    $$\|\funk-\est\|_\infty=\mathcal{O}\Big(\Big(\frac{\log(n)}{n}\Big)
    ^{\frac{v-1}{2(v+3)}}\Big)$$
 from the first part of Theorem \ref{theorem1} yields
    $$\rho\lb (t,0), (t, d_n) \rb = \|\mathcal{R}[\est-\funk] \|_\infty \le
    \|\est-\funk\|_\infty = \mathcal{O}\Big(\Big(\frac{\log(n)}{n}\Big)
    ^{\frac{v-1}{2(v+3)}}\Big) ~~~a.s. $$
 such that for a sequence $\delta_n \downarrow 0$,
 say e.g. $\delta_n = \log(n)^{-1}$
    $$\mathbb{P}\Big( \sup_{t \in \mathbb{R}}\rho\lb (t,0), (t, d_n) \rb <
    \delta_n \Big) \to 1. $$
 Combining these considerations with the right side of
 \eqref{lemmathm2est1} yields that
 $n^{1/2}M_n(t)=o_P(1)$ uniformly in $t$, proving the Lemma provided that
 \eqref{equicontinuity} holds.\\

 This statement is a consequence of Lemma A.19 in \cite{N2006}
, which requires four regularity properties of the process under consideration. The rest
 of the proof consists in verifying these properties.

\begin{enumerate}
\item For all $\eta>0$ we have to show:
\begin{eqnarray*}
    \sum_{\k \in \K} \mathbb{E} \Big( \sup_{(t,d) \in \mathcal{F}} \la
    Z_{n, \mathbf{k}}(t,d)\ra \mathbbm{1} \left\{ \la Z_{n, \mathbf{k}}
    (t,d)\ra > \eta \right\} \Big) \to 0.
\end{eqnarray*}
 This is easy to see, since $|Z_{n,\mathbf{k}}| \le C n^{-1/2}$ (recall that
 $\max_\mathbf{k} \weight \le C n^{-1}$) and so the sum is equal to
 $0$ for all $n$ larger than some $n_0$.

\item For every sequence $\delta_n \downarrow 0$
\begin{eqnarray}
    \sup_{\rho((t,d),(\tilde{t}, \tilde{d}))< \delta_n} \Big| \sum_{\k \in \K}
    \mathbb{E}\Big[   \Big( Z_{n, \mathbf{k}}(t,d)- Z_{n, \mathbf{k}}(\tilde{t},
    \tilde{d}) \Big)^2 \Big] \Big| \to 0. \label{Neumeyercond2}
\end{eqnarray}
 Consider the expectation for some fixed but arbitrary $\k \in \K$
 which can be bounded uniformly as follows:
\begin{eqnarray*}
    &  &  \mathbb{E}\left[   \lb Z_{n, \mathbf{k}}(t,d)- Z_{n, \mathbf{k}}
    (\tilde{t},\tilde{d}) \rb^2 \right] \\[1ex] & \le &
    C n^{-1} \left[ |\cdf(t+d(z_\mathbf{k}))-\cdf(\tilde{t}+ \tilde{d}
    (z_\mathbf{k}))|+ |\cdf(t)-\cdf(\tilde{t})| \right] \\[1ex] & \le &
    C n^{-1} \Big[ |\cdf(t+d(z_\mathbf{k}))-\cdf(\tilde{t}+ \tilde{d}
    (z_\mathbf{k}))|+ \delta_n \Big] \\[1ex] & \le &
    C n^{-1} \Big[ |\cdf(t+d(z_\mathbf{k}))-\cdf(t+ \tilde{d}(z_\mathbf{k}))|
    + ~|\cdf(t+\tilde{d}(z_\mathbf{k}))-\cdf(\tilde{t}+ \tilde{d}(z_\mathbf{k}))|
    + \delta_n \Big].
\end{eqnarray*}
 All three terms inside the square brackets are uniformly of order
  $o(1)$. This can be shown as follows: An application of the mean value theorem
  demonstrates that the first term is a null sequence:
    $$ |\cdf(t+d(z_\mathbf{k}))-\cdf(t+ \tilde{d}(z_\mathbf{k}))| \le
    \|\dens\|_\infty \|d-\tilde{d}\|_\infty \le C \delta_n \to 0.$$
 The middle term is bounded by $\delta_n$ by definition of our semimetric
 $\rho$ in \eqref{semimetric}, when we consider that $\tilde{d}\in \R(\smoothel2)$
 and therefore $\|\tilde{d}\|_\infty\le1$. Consequently it is $o(1)$, as well as the
 last term by assumption.

\item Denoting the $L_2^n$-bracketing number, as given in Definition
 \ref{L_2^Ndef}, by $N_{[]}^{L_2^n}(\epsilon, \mathcal{F})$, the condition
 we have to check next is, that for every sequence $\delta_n \downarrow 0$:
\begin{equation}
    \int_0^{\delta_n} \sqrt{\log \big( N_{[]}^{L_2^n}(\epsilon, \mathcal{F})\big) }
    d\epsilon \to 0. \label{Neumeyercond3}
\end{equation}
 For the construction of an adequate partition of $\mathcal{F}$ satisfying
 \eqref{L_2^N}, consider the $\epsilon^2$-brackets
 $[g_j^L, g_j^U], ~j=1,...,J=\mathcal{O}(\exp(\epsilon^{4/(\smoothindex-t)}))$
 of $ \smoothel2 $,  where  $t>3$ such that $\smoothindex-t>2$ (note that $\smoothindex>5$ by assumption). The images of these brackets under $\mathcal{R}$ are
 simply $[\R g_j^L, \R g_j^U]$, due to monotonicity of the integral and they are
 still $\epsilon^2$-brackets, since $\R$ reduces $\| \cdot \|_\infty$-distance. As
 a consequence we receive $\epsilon^2$-brackets $[\R g_j^L, \R g_j^U]$ of the
 whole class $\R (\smoothel2)$.\\
 Additionally choose $y_{i}^L<y_{i}^U$ with $i=1,...,I=\mathcal{O}(\epsilon^{-2})$,
 such that the intervals $[y_{i}^L, y_{i}^U]$ form a partition of the real line
 (for infinite values we take the intervals to be half open), and such that  each
 interval has probability mass $\le \epsilon^2$. Then the sets
\begin{eqnarray}
    \mathcal{F}_{i, j, \epsilon}^n:=[y_{i}^L, y_{i}^U] \times [ \R g_j^L,  \R g_j^U]
\end{eqnarray}
form a partition of $\mathcal{F}$. Their number is of order
$\mathcal{O}(\exp(\epsilon^{4/(\smoothindex-t)})$, where we might have to
slightly shrink $t$ such that still $t>3$ and $\smoothindex-t>2$ hold. Now we have
to show that \eqref{L_2^N} holds, that is in the present case for an arbitrary
$ \mathcal{F}_{i, j, \epsilon}^n$
    $$ \sum_{\k \in \K} \mathbb{E} \Big[ \sup_{(t,d), (\tilde{t}, \tilde{d})
    \in \mathcal{F}_{i, j, \epsilon}^n} \lb Z_{n,\mathbf{k}}(t,d) -Z_{n,\mathbf{k}}
    (\tilde{t}, \tilde{d})\rb^2 \Big] \le \epsilon^2. $$
 In the subsequent calculation we define the expressions $\cdf(\pm \infty)$ and
 $\dens(\pm \infty)$ by taking the respective limits. The left side of the above
 inequality is bounded  by
\begin{eqnarray*}
    & & C n^{-1} \sum_{\k \in \K} \mathbb{E} \Big[ \sup_{(t,d)(\tilde{t}, \tilde{d})
    \in\mathcal{F}_{i, j, \epsilon}^n} \Big( \la \mathbbm{1}\{\varepsilon_\mathbf{k}
    \le t+d(z_\mathbf{k}) \}- \mathbbm{1}\{\varepsilon_\mathbf{k}\le \tilde{t}+
    \tilde{d}(z_\mathbf{k})\} \ra  \\[1ex]  & & ~~~~+~ \la \mathbbm{1}
    \{\varepsilon_\textbf{k} \le t\}-\mathbbm{1}\{\varepsilon_\textbf{k} \le
    \tilde{t}\}  \ra \Big)^2 \Big] \\[1ex]
    & \le & C n^{-1}\sum_{\k \in \K} \mathbb{E} \Big[ \la \mathbbm{1}
    \{\varepsilon_\mathbf{k}\le y_{i}^U+\R g_j^U(z_\mathbf{k}) \}- \mathbbm{1}
    \{\varepsilon_\mathbf{k}\le  y_{i}^L+\R g_j^L(z_\mathbf{k})\}\ra \\[1ex]
    &  &~~~~ +~\la \mathbbm{1}\{\varepsilon_\textbf{k} \le y_{i}^U\}-\mathbbm{1}
    \{\varepsilon_\textbf{k} \le y_{i}^L\}  \ra \Big] \\[1ex]
     & \le & C n^{-1} \sum_{\k \in \K} \Big[ |\cdf( y_{i}^U+\R g_j^U(z_\mathbf{k}))-
     \cdf(y_{i}^L+\R g_j^L(z_\mathbf{k}))|
    \\[1ex] & &~~~~ +~ C  \la \cdf \lb y_{i}^U \rb -\cdf \lb y_{i}^L \rb \ra  \Big]
    \\[1ex]  & \le &
    C n^{-1}\sum_{\k \in \K} |\cdf( y_{i}^U+\R g_j^U(z_\mathbf{k}))-\cdf(y_{i}^L
    +\R g_j^L(z_\mathbf{k}))|    + C\epsilon^2 \\[1ex]
    & \le & C n^{-1}\sum_{\k \in \K} \{ |\cdf( y_{i}^U+\R g_j^U(z_\mathbf{k}))-
    \cdf(y_{i}^L+\R g_j^U(z_\mathbf{k}))| \\[1ex]
    & & ~~~~ + ~|\cdf(y_{i}^L+\R g_j^U(z_\mathbf{k}))-\cdf(y_{i}^L+
    \R g_j^L(z_\mathbf{k}))|   \} + C\epsilon^2 \\[1ex]
    & \le & C \lb\epsilon^2 +\|\dens(y)\|_\infty ~\|\R g_j^U-
    \R g_j^L\|_\infty + \epsilon^2 \rb \le C \epsilon^2.
\end{eqnarray*}
 Replacing $\epsilon$ by $\epsilon  C^{-1/2}$ yields the desired result, without
 changing the rate of the upper bound $\mathcal{O}(\exp(\epsilon^{4/(\smoothindex-t)})$
 of the $L_2^n$-bracketing number. Thus the integral in \eqref{Neumeyercond3} converges
 since $\smoothindex-t>2$.

\item Finally we have to prove that $(\mathcal{F}, \rho)$ is  totally bounded. By definition
 $\rho$ is a maximum semimetric defined on the product space
 $\mathbb{R} \times \R(\smoothel2)$.  Hence it suffices to show that each of the  spaces
 $(\mathbb{R}, \rho_1)$, $(\R(\smoothel2), \rho_2)$ is totally bounded, where we define for $t, \tilde{t}\in \mathbb{R} $ and $d_1, d_2 \in \R(\smoothel2)$
\begin{eqnarray*}
 \rho_1(t,\tilde{t})    & := &\sup_{x \in [-1,1]}|\cdf(t+x)-\cdf( \tilde{t}+x)| ~~   \\[1ex]
    \rho_2(d_1, d_2)&   :=& \|d_1-d_2\|_\infty.
 \end{eqnarray*}
 We start with $\rho_1$ and demonstrate that for every $\epsilon>0$ we can find a finite number
 of $t_1,...,t_J \in \mathbb{R}$ such that for every $t \in \mathbb{R}$ there exists a $t_j$ such that
\begin{equation}
    \sup_{x \in [-1,1]}|\cdf(t+x)-\cdf(t_j+x)|\le \epsilon.
\end{equation}
 Let $M:= \max_{t \in \mathbb{R}}|\dens(t)|$ and $I$ be a closed interval with probability mass
 larger than $1-\epsilon$. Take an equidistant grid with maximal width $\epsilon/M$ of points $t_j$ for
 $j=1,...,J$ across $I$ (including the boundary points) and now let, for an
 arbitrary $t \in \mathbb{R}$ say $t_j$ be one of the closest points to $t$ of this grid. If $t \notin I$
 we choose a boundary point of $I$ and the result is immediate. If $t \in I$ we get by the mean value
 theorem:
\begin{equation*}
    |\cdf(t+x)-\cdf(t_j+x)|\le\epsilon \|\dens\|_\infty M^{-1}=\epsilon.
\end{equation*}
 For $\rho_2$ we recall that by our above observations for every $\epsilon>0$ the bracketing
 number of $\R(\smoothel2)$ with respect to the norm $\|\cdot\|_\infty$ is finite and thus in particular
 we have total boundedness.
\end{enumerate}

\noindent Having established these regularity properties, by Neumeyer's Lemma A.19 (2006) equicontinuity
follows which completes the proof of Lemma \ref{lemma3}.

\end{proof}

\noindent Besides Lemma \ref{lemma3} we require some additional approximation results for a proof of Theorem \ref{theorem2}.

\begin{proposition} \label{proposition6}
 Under the assumptions of Theorem \ref{theorem2} we have

\begin{equation}
    \sum_{\k \in \K} \weight \mathbb{P}\lb \resestk \le t \rb - \int_\mathcal{D}\cdf \lb t
    + \mathcal{R}[\est-\funk](z)\rb d\lambda(z)= o_P\lb n^{-1/2} \rb \label{prop6eq1}
\end{equation}

\begin{eqnarray}
     & & \int_\mathcal{D} \cdf(t+\mathcal{R}[\est - \funk](z)) d\lambda(z)-\cdf(t)
     \label{prop6eq2} -\dens(t) \int_\mathcal{D} \mathcal{R}[\est-\funk](z)
     d\lambda(z)   =   o_P\lb n^{-1/2}  \rb
\end{eqnarray}

\begin{equation}
    \int_\mathcal{D} \mathcal{R} [\est-\funk](z) d\lambda(z)-
    \sum_{\k \in \K}\varepsilon_
    \mathbf{k} \weight = o_P \lb n^{-1/2}  \rb \label{prop6eq3}
\end{equation}

\end{proposition}

\begin{proof}[Proof of Proposition \ref{proposition6}]
 Recalling the definitions of the estimated residuals
 $\resestk:=\varepsilon_\mathbf{k}-\mathcal{R}[\est-\funk](\desk)$ and the
 weights $\weight:=\lambda(\GSk)$, we begin by
 rewriting the left side of  \eqref{prop6eq1}
\begin{equation*}
    \sum_{\k \in \K} \int_{\GSk} \cdf \lb t+\mathcal{R}[\est-\funk](\desk) \rb - \cdf
    \lb t+\mathcal{R}[\est-\funk](z) \rb d\lambda(z).
\end{equation*}
 According to the mean value theorem the absolute of this term is bounded by
\begin{equation*}
    \sum_{\k \in \K}  \int_{\GSk} \la \left\{ \mathcal{R}[\est-\funk](\desk) -
    \mathcal{R}[\est-\funk](z)\right\} \dens(t_z) \ra d\lambda(z),
\end{equation*}
 where  $t_z$ is some suitable point between $t+\mathcal{R}[\est-\funk](z)$ and
 $t+\mathcal{R}[\est-\funk](\desk)$. Since the density is bounded it suffices to
 show that
\begin{equation*}
    \sup_{z \in \GSk}\la  \mathcal{R}[\est-\funk](\desk) -\mathcal{R}[\est-\funk](z)\ra
    = o_P(n^{-1/2} ).
\end{equation*}
 An application of Cauchy-Schwarz yields
\begin{equation*}
    \sup_{z \in \GSk}\la  \mathcal{R}[\est-\funk](\desk) -\mathcal{R}[\est-\funk](z)\ra
    \le 2 \| \nabla \mathcal{R}[\est-\funk]\|_\infty \|z-\desk\|_2.
\end{equation*}
 By Assumption \ref{assumption1} $\|z-\desk\|_2 = \mathcal{O}(1/\sqrt{n})$.
  Moreover by
 Corollary \ref{corollary2} the gradient $\nabla \mathcal{R}[\est-\funk]$ converges
 uniformly to $0$. Thus we get the desired result. The estimate \eqref{prop6eq2} follows by similar
 arguments, while
 \eqref{prop6eq3} is based on two observations. Firstly, since
 $\psi_{(0,0)} =1$ we can rewrite the integral
\begin{equation*}
    \int_\mathcal{D} \mathcal{R} [\est-\funk](z) d\lambda(z)= \lbp
    \mathcal{R} [\est-\funk], \psi_{(0,0)}\rbp_{\mathcal{L}^2(\mathcal{D}, \lambda)}= \hat{R}(0,0)
     - R(0,0).
\end{equation*}
 Secondly as the errors are centered
\begin{equation*}
    \sum_{\k \in \K} \weight \varepsilon_\k= \hat{R}(0,0)-\mathbb{E}\Big[
    \hat{R}(0,0) \Big].
\end{equation*}
 Combining these results yields the  representation
 $\mathbb{E}[ \hat{R}(0,0) ]-R(0,0)$ for the left side of \eqref{prop6eq3}.
 By Proposition \ref{proposition4} this difference is of order $\mathcal{O}(n^{-1})$.

\end{proof}

 Equipped with our observations in Lemma \ref{lemma3} and Proposition \ref{proposition6}
 Theorem \ref{theorem2} is easily deduced:

\noindent \textbf{Proof of Theorem \ref{theorem2}}
 We apply the triangular inequality to arrive at the following decomposition
\begin{eqnarray*}
    & & \sup_{t \in \mathbb{R}} \Big| \sum_{\k \in \K} \weight\left[ \mathbbm{1}
    \{\resestk\le t\} -\mathbbm{1}\{\varepsilon_\mathbf{k}\le t\} -
    \varepsilon_\mathbf{k}\dens(t) \right] \Big|  \\[1ex]
    & \le & \sup_{t \in \mathbb{R}} \Big|  \sum_{\k \in \K} \weight \left[
    \mathbbm{1}\{\resestk\le t\} -\mathbbm{1}\{\varepsilon_\mathbf{k}\le t\} +
    \cdf(t)-\mathbb{P} \lb \resestk \le t \rb \right] \Big| \\[1ex]
    & + & \sup_{t \in \mathbb{R}} \Big| \sum_{\k \in \K} \weight \mathbb{P} \lb \resestk
    \le t \rb - \int_\mathcal{D}\cdf \lb t+\mathcal{R}[\est-\funk](z) \rb d\lambda(z)
    \Big| \\[1ex]
    & + & \sup_{t \in \mathbb{R}} \Big| \int_\mathcal{D} \cdf \lb t+\mathcal{R}
    [\est-\funk](z) \rb d\lambda(z) -\cdf(t) - \dens(t) \int_\mathcal{D}\mathcal{R}
    [\est-\funk](z) d\lambda(z) \Big| \\[1ex]
    & + & \sup_{t \in \mathbb{R}} \Big| f_\varepsilon(t) \int_\mathcal{D}\mathcal{R}
    [\est-\funk](z) d\lambda(z) - \sum_{\k \in \K} \weight \varepsilon_\mathbf{k}\dens(t)
    \Big|.
\end{eqnarray*}
 Each of the terms on the right side is of order $o_{P} \lb n^{-1/2} \rb$, the first one
 by Lemma \ref{lemma3} and the other ones by Proposition \ref{proposition6}. \qed \\

\subsection{Proof of Corollary \ref{corollary2}} This is a consequence of Theorem \ref{theorem2}, as we can
 represent the process $\sqrt{n}(\cdf-\estcdf)$ as  sum of independent stochastic
 processes and a negligible term:
\begin{eqnarray*}
    \sqrt{n}(\cdf-\estcdf)&  =& \sum_{\mathbf{k}\in \K} \sqrt{n} \weight \left\{ \mathbbm{1}
    \{\resestk \le t\} - \cdf(t) \right\}  \\[1ex]
    & =  & \sum_{\mathbf{k}\in \K} \sqrt{n} \weight \left\{ \mathbbm{1}\{\varepsilon_\mathbf{k}
    \le t\} - \cdf(t) + \varepsilon_\mathbf{k} \dens(t) \right\} + o_P(1).
\end{eqnarray*}
 The sum on the right side converges to a Gaussian process, by application of a
 functional CLT for triangular arrays found in \cite{N2006}.

\section{Auxiliary results} \label{ApB}

\subsection{Uniform bounds}
 We begin stating some frequently used properties of the radial polynomials  which are taken from
 \cite{BoWo1970} and \cite{Janssen:14}.

\begin{proposition} \label{Ap1}
\begin{enumerate}
    \item For all $(l,m)\in \mathcal{N}$
        \begin{equation}
            \sup_{0 \le r \le 1}|R_m^{|l|}(r)| = 1 .\label{Rad1}
        \end{equation}
    \item For all $(l,m), (l,m') \in \mathcal{N}$
        \begin{equation}
            \int_0^1 \sqrt{2(m+1)}R_m^{|l|}(r) \sqrt{2(m'+1)}R_{m'}^{|l|}(r) r ~dr
            = \delta_{m, m'}. \label{Rad2}
        \end{equation}
    \item For all $(l,m)\in \mathcal{N}$ the derivative of the corresponding
             radial polynomial has the
           following structure:
         \begin{eqnarray}
            \frac{d}{dr} R_m^{|l|}(r) &= & \sum_{j=0}^{\frac{1}{2}(m-1-|l-1|)}
            (m-2j)R_{m-1-2j}^{|l-1|}(r)  \label{Rad3}\\ &+& \sum_{j=0}^{\frac{1}{2}
            (m-1-|l+1|)}(m-2j)R_{m-1-2j}^{|l+1|}(r)\nonumber.
        \end{eqnarray}
\end{enumerate}
\end{proposition}

 \noindent Next we provide upper bounds on the $\|\cdot\|_\infty$-norm of the derivatives
 of the Chebychev and radial polynomials. The bounds on the radial polynomials follow by
 the above Proposition and the bounds for the Chebychevs by identities from \cite{HaMa1986}.\\

\begin{proposition} \label{proposition3}
 Let $k \in \mathbb{N}_0$ and $(l,m) \in \mathcal{N}$, then
\begin{eqnarray*}
    \sup_{0 \le r \le 1} \la \frac{d^k}{dr^k} R_m^{|l|}( r) \ra \le  m^{2k}
\end{eqnarray*}
and
\begin{eqnarray*}
    \sup_{0 \le s \le 1} \la \frac{d^k}{d s^k} U_m(s) \ra \le  (m+1)m^{2k}.
\end{eqnarray*}

\end{proposition}

\begin{proof}[Proof of Proposition \ref{proposition3}]
 In order to show the first statement, we apply the identities \eqref{Rad1} and
 \eqref{Rad3} from Proposition \ref{Ap1} and use an induction argument. The initial step
 is given by \eqref{Rad1} and the induction hypothesis is
\begin{eqnarray*}
    \sup_{0 \le r \le 1} \la \frac{d^k}{dr^k} R_m^{|l|}( r) \ra \le  m^{2k}.
\end{eqnarray*}
By virtue of \eqref{Rad3} we have
\begin{eqnarray*}
    \la \frac{d^{k+1}}{dr^{k+1}}  R_m^{|l|}( r) \ra & = & \bigg| \sum_{j=0}^{\frac{1}{2}
    (m-1-|l-1|)}(m-2j) \frac{d^{k}}{dr^{k}}R_{m-1-2j}^{|l-1|}(r) \\[1ex]
    & & + \sum_{j=0}^{\frac{1}{2}(m-1-|l+1|)}(m-2j)\frac{d^{k}}{dr^{k}}R_{m-1-2j}^{|l+1|}
    (r) \bigg| \\[1ex]
    & \le & 2  \sum_{j=0}^{\frac{1}{2}(m-1-|l-1|)}(m-2j) m^{2k}  \le  m^{2k+2},
\end{eqnarray*}
 where we have used the induction hypothesis to bound the derivatives of
 $R_m^{|l|}$. \\[1ex]
 The case of the Chebychev polynomials is similar. In order to prove the second
 identity in Proposition \ref{proposition3} we cite a few well known facts about
 Chebychev polynomials from \cite{HaMa1986}
\begin{enumerate}
    \item For all $m \in \mathbb{N}$ $U_m$ is uniformly bounded by $m+1$.
    \item Let $T_m$ denote the Chebychev polynomial of the first kind, which satisfies the
          differential equation
        \begin{equation*}
            \frac{d}{ds}T_m(s)= U_{m-1}(s) m. \label{masonident1}
        \end{equation*}
           For all $m\in \mathbb{N}$ $T_m$ is uniformly bounded by $1$.
    \item For all $m \in \mathbb{N}$ the representation
        \begin{equation*}
            \frac{d}{ds} U_m(s) =  \sum_{j=0}^{\bullet m-2} \frac{(m^2-j^2)m}{m+1} T_j(s)
            \label{masonident2}
        \end{equation*}
         holds, where  $\bullet$ indicates that we only sum over such terms where $m-j$
         is even.
\end{enumerate}
 The proof now follows by an induction, analogous to that of the first part.
\end{proof}

\subsection{Proof of Proposition \ref{proposition1}} \label{ApB2}
 We employ these bounds to sketch a proof of Proposition \ref{proposition1}. The
 techniques are borrowed from the theory of Fourier series. It is
 well known that a continuous function $f$ on a compact interval, with absolutely
 summable Fourier coefficients is identical to its Fourier series $f_\infty$. This is
 most easily proven by observing that $f$ and $f_\infty$ are identical in mean and
 that by uniform  convergence  $f_\infty$ is also continuous. We proceed
 analogously for the proof  of the  identities \eqref{L2expansiong} - \eqref{inverse}.
 The differentiability is an immediate  consequence of this argument. To avoid redundancy we confine our investigation to equation \eqref{L2expansiong}.

Firstly we define the function on
 the right side of \eqref{L2expansiong} by $\tilde{g}$. Obviously
\begin{equation}
    \int_\mathcal{B}(\funk-\tilde{\funk})^2d\mu=0.  \label{integralnull}
\end{equation}

 \noindent As $\mu$ is absolutely continuous with respect to the Lebesgue measure
 the set $\{g=\tilde{g}\}$ has Lebesgue measure $0$ and thus \eqref{L2expansiong}
 follows if we can establish the  continuity of $\tilde{g}$ (recall that $\funk$ is
 continuous by assumption). Continuity of $\tilde{\funk}$ is implied by the uniform
 convergence of the sequence of continuous functions
\begin{equation}
    \Big( \sum_{m =0}^N \sum_{l=-m}^m \varphi_{(l,m)} \lbp  \funk, \varphi_{(l,m)}
    \rbp_{\mathcal{L}^2(\mathcal{B}, \mu)}\Big)_{N \in \mathbb{N}}
     \label{functionsequence}
\end{equation}
 to $\tilde{\funk}$ for $N \to \infty$.  To see this we consider the difference
\begin{eqnarray*}
    & & \Big\| \tilde{g}- \sum_{m =0}^N \sum_{l=-m}^m \varphi_{(l,m)} \lbp  \funk, \varphi_{(l,m)}
    \rbp_{\mathcal{L}^2(\mathcal{B}, \mu)} \Big\|_\infty =
    \Big\|  \sum_{m =N}^\infty \sum_{l=-m}^m \varphi_{(l,m)} \lbp  \funk, \varphi_{(l,m)}
    \rbp_{\mathcal{L}^2(\mathcal{B}, \mu)}  \Big\|_\infty \\[1ex]
    & \le &  \sum_{m =N}^\infty \sum_{l=-m}^m \Big\| \varphi_{(l,m)}\Big\|_\infty
    \la \lbp  \funk, \varphi_{(l,m)}
    \rbp_{\mathcal{L}^2(\mathcal{B}, \mu)}  \ra \le
    \sum_{m =N}^\infty \sum_{l=-m}^m \sqrt{m+1}  \big| \lbp  \funk, \varphi_{(l,m)}
    \rbp_{\mathcal{L}^2(\mathcal{B}, \mu)}  \big|
\end{eqnarray*}
 where we used \eqref{varphibound} in the last step. Plugging  the
 identity \eqref{inverse} (recall that we  already know it in an $\mathcal{L}^2$-sense from
 equation \eqref{inverseL2}) into the
 inner products yields
\begin{equation}
    \sum_{m =N}^\infty \sum_{l=-m}^m  \sqrt{m+1} \big| \lbp  \funk, \varphi_{(l,m)}
    \rbp_{\mathcal{L}^2(\mathcal{B}, \mu)}  \big| =  \sum_{m=N}^\infty \sum_{l=-m}^{ m}
    (m+1)\big|  \left<
    \R \funk, \psi_{(l,m)} \right>_{\mathcal{L}^2(\mathcal{D}, \lambda)}
    \big|.
\end{equation}
 By the series condition in \eqref{ellipsoidondition}, the right and thus the left side
 converge to $0$, which proves continuity of $\tilde{g}$. Consequently, it follows from
 \eqref{integralnull}, that $\funk=\tilde{g}$.

 To establish differentiability of $\funk$ and $\mathcal{R}\funk$ we use their
 $\mathcal{L}^2$-representations \eqref{L2expansiong} and  \eqref{L2expansionRg}.
 Differentiability and summation may be interchanged by uniformity arguments, using the
 bounds from Proposition \ref{proposition3}. Continuity of the derivatives is then
 derived as in the above argumentation.

\subsection{Proof of Proposition \ref{proposition4}} \label{ApB3} By definition of $\rpN$ in
\eqref{R_N} and the weights $\weight=\lambda(\GSk)$ we obtain:
\begin{eqnarray}
    \label{intgsk} \big| \rp -\mathbb{E}\rpN\big|  & \le & \sum_{\k \in \K}\la \int_{\GSk}
     \overline{\psi_{(l,m)}}(z)    \mathcal{R}\funk(z)- \overline{\psi_{(l,m)}}(z_\textbf{k})
     \mathcal{R}\funk(z_\textbf{k})      d\lambda(z)\ra. \nonumber \\[1ex]
    & \le & \sum_{\k \in \K} \la \int_{\GSk}  \mathfrak{Re}(\psi_{(l,m)}(z))
    \mathcal{R}\funk(z)-       \mathfrak{Re}(\psi_{(l,m)}(z_    \textbf{k}))\mathcal{R}
    \funk(z_\textbf{k}) d\lambda(z)\ra \nonumber\\[1ex]
    & +& \sum_{\k \in \K}\la \int_{\GSk} \mathfrak{Im}(\psi_{(l,m)}(z))\mathcal{R}
    \funk(z)   -\mathfrak{Im}(\psi_{(l,m)}(z_
    \textbf{k}))\mathcal{R}\funk(z_\textbf{k}) d\lambda(z)\ra.\nonumber
\end{eqnarray}
 By Proposition \ref{proposition1} the function $\mathcal{R}\funk$ is twice continuously
 differentiable. Recalling the definition of $\psi_{(l,m)}$, we observe that  the real part
\begin{equation}
    \mathfrak{Re}(\psi_{(l,m)}(s, \phi))=U_{m}(s)\cos(\phi l)
\end{equation}
 and the imaginary part
 \begin{equation}
    \mathfrak{Im}(\psi_{(l,m)}(s, \phi))=U_{m}(s)\sin(\phi l)
 \end{equation}
 are infinitely often differentiable. By Proposition \ref{proposition3} it is now easy to
 see, that all second order derivatives of these functions are uniformly bounded by $2 m^5$.
 We now  use a Taylor expansion and obtain for any $\mathbf{k}=(k_1, k_2)$
\begin{eqnarray*}
    R(\k) & :=&\int_{\GSk} \mathfrak{Re}(\psi_{(l,m)}(z))\mathcal{R}\funk(z) -
    \mathfrak{Re}(\psi_{(l,m)}(z_\textbf{k}))\mathcal{R}\funk(z_    \textbf{k})
     d\lambda(z) \\[1.2ex]
    &= & \int_{k_1/d}^{(k_1+1)/d} \int_{2\pi k_2/n}^{2\pi(k_2+1)/n} (s-z_{k_1})\frac{d}
    {ds}\lbr\mathfrak{Re}(\psi_{(l,m)})\mathcal{R}\funk \rbr(z_\textbf{k})2 \pi^{-1} \sqrt{1-
    s^2} d\phi ds\\[1.2ex]
    & + & \int_{k_1/d}^{(k_1+1)/d} \int_{2\pi k_2/n}^{2\pi(k_2+1)/n} (\phi-z_{k_2})
    \frac{d}{d\phi}\lbr\mathfrak{Re}(\psi_{(l,m)})\mathcal{R}\funk \rbr(z_\textbf{k})2 \pi^{-1}
    \sqrt{1-s^2} d\phi ds\\[1.2ex]
    & + & \int_{k_1/d}^{(k_1+1)/d} \int_{2\pi k_2/n}^{2\pi(k_2+1)/n} (s-z_{k_1})(\phi-
    z_{k_2})  \frac{d}{d\phi}\frac{d}{ds}\lbr\mathfrak{Re}(\psi_{(l,m)})\mathcal{R}\funk
    \rbr(\xi_1) 2 \pi^{-1} \sqrt{1-s^2} d\phi ds\\[1.2ex]
    & + & \int_{k_1/d}^{(k_1+1)/d} \int_{2\pi k_2/n}^{2\pi(k_2+1)/n} 2^{-1}(s-
    z_{k_1})^2\frac{d^2}{ds^2}\lbr\mathfrak{Re}(\psi_{(l,m)})\mathcal{R}\funk \rbr(\xi_2) 2
    \pi^{-1} \sqrt{1-s^2} d\phi ds\\[1.2ex]
    & + & \int_{k_1/d}^{(k_1+1)/d} \int_{2\pi k_2/n}^{2\pi(k_2+1)/n} 2^{-1}(\phi-
    z_{k_2})^2\frac{d^2}{d\phi^2}\lbr\mathfrak{Re}(\psi_{(l,m)})\mathcal{R}\funk \rbr(\xi_3) 2
    \pi^{-1} \sqrt{1-s^2} d\phi ds.
\end{eqnarray*}
 Here $\xi_1, \xi_2, \xi_3$ denote points dependent on $s$, $\phi$ and $z_\textbf{k}$,
 which are located inside $\GSk$ because of its convexity. The first two
 integrals vanish because of the choice of our design points. Moreover
 $|s-z_{k_1}|$ and $|\phi-z_{k_2}|$ are bounded by $C n^{-1/2}$ by  Assumption 1. The
 second order derivatives of $\mathcal{R}\funk $  are bounded (because they are
 continuous) and those of $\mathfrak{Re}(\psi_{(l,m)})$ are bounded by  $2 m^5$, as we
 have noted above. Thus the term $R(\k)$ is of order $\mathcal{O}(m^5n^{-1})$. Treating
 the  integrals in the sum over the imaginary parts in same fashion yields the result. \qed

\subsection{Proof of Proposition \ref{proposition5}} \label{ApB4}

 We begin by rewriting the series condition
 \eqref{octahedralcondition2} as
\begin{equation}
    1\ge \sum_{m=0}^\infty \sum_{l=-m}^m (m+1)^\smoothindex|\rp|=
     \sum_{m=0}^\infty \sum_{l=-m}^m (m+1)^{\smoothp} \la
     \lbp  \funk, \varphi_{(l,m)}
    \rbp_{\mathcal{L}^2(\mathcal{B}, \mu)} \ra,
      \label{ellipserewritten}
\end{equation}
 where we define $\smoothp:=\smoothindex-1/2$ for convenience of notation.
 The reason  for this modification is that all conditions are now  expressed directly
 by $\funk$ instead of its Radon transform. \\

 Our proof rests upon an observation found in the monograph \cite{VW1996}.
 If we can find suitable functions $g_1,...,g_L$ with finite  $\| \cdot \|_\infty$-norm,
 such that the class $\smoothel2 $ is included in the union of the
 $\|\cdot\|_\infty$-balls with radius $\epsilon$, i.e.
\begin{equation}
    \smoothel2 \subset U_\epsilon^{\| \cdot \|_\infty}(g_1) \cup ... \cup
    U_\epsilon^{\| \cdot \|_\infty}(g_L), \label{covering}
\end{equation}
then the $\|\cdot\|_\infty$-bracketing number of  $\smoothel2$ for
$2\epsilon$ is upper bounded by $L$. The corresponding  brackets are then simply given
by $[g_l-\epsilon, g_l+\epsilon]$ for all $l \in \{1,...,L\}$. We will thus confine ourselves
to showing that the covering number of $\smoothel2$ for some arbitrary but fixed
$\epsilon>0$ is upper  bounded by
$L=L(\epsilon) \le \exp ( C \epsilon^{-2/(\smoothp-\smoothpp)})$, where
$\smoothpp:=t-1/2$.\\

 The rest of the proof consists of the construction of such a class of functions,
 breaking up  $\smoothel2$ in $\epsilon$-balls and verifying that their number
 is bounded in the desired way.  We begin by relating closeness of Radon coefficients
 to closeness in $\|\cdot\|_\infty$-norm.

 Invoking Proposition \ref{proposition1}, we observe that every function
 $\funk \in  \smoothel2 $ is identical to its $\mathcal{L}^2$-expansion
\begin{equation*}
    \funk=\sum_{m = 0 }^\infty \sum_{l=-m}^m\varphi_{(l, m)}\lbp  \funk, \varphi_{(l,m)}
    \rbp_{\mathcal{L}^2(\mathcal{B}, \mu)}  .
\end{equation*}
Because of \eqref{ellipserewritten} and $\|\funk\|_\infty \le 1$
we get for each $\funk \in \smoothel2$
\begin{equation}
    \big|  \lbp  \funk, \varphi_{(l,m)}
    \rbp_{\mathcal{L}^2(\mathcal{B}, \mu)}
     \big| \le \frac{1}{ (m+1)^{\smoothp}} \quad \forall (l,m)\in \mathcal{N}. \label{cof}
\end{equation}
We will now investigate the distance between two functions $\funk$, $\tilde{g}$ in $ \smoothel2$
which have similar Radon coefficients in the  sense that
\begin{equation}
    \big| \lbp  \funk, \varphi_{(l,m)}
    \rbp_{\mathcal{L}^2(\mathcal{B}, \mu)}  -
    \left<\tilde{g}, \varphi_{(l, m)} \right>_{\mathcal{L}^2(\mathcal{B}, \mu)}
     \big| \le \frac{\epsilon}{C (m+1)^{\smoothpp}} \quad \forall (l,m)\in \mathcal{N}, \label{16}
\end{equation}
 for some $\epsilon>0$. For sufficiently large $C>0$, depending on ${\smoothpp}$ only,
 the maximal distance between $\funk$ and $\tilde{g}$ can be bounded via
\begin{eqnarray*}
    \|\funk-\tilde{g}\|_\infty & \le & \sum_{m = 0 }^\infty \sum_{l=-m}^m
     \la \left< \funk, \varphi_{(l, m)} \right>_{\mathcal{L}^2(\mathcal{B}, \mu)}
     - \left< \tilde{g}, \varphi_{(l, m)} \right>_{\mathcal{L}^2(\mathcal{B}, \mu)}
     \ra \|\varphi_{(l, m)}\|_\infty \\[1ex]
    & \le & \sum_{m = 0 }^\infty\sum_{l=-m}^m \sqrt{m+1}  \la
    \left<\funk,  \varphi_{(l, m)} \right>_{\mathcal{L}^2(\mathcal{B}, \mu)}
    - \left< \tilde{g}, \varphi_{(l, m)}\right>_{\mathcal{L}^2(\mathcal{B}, \mu)}
    \ra  \\[1ex] & \le & \frac{\epsilon}{C}  + \sum_{m = 1 }^\infty  (m+1)^{3/2}
    \frac{\epsilon}{C (m+1)^{\smoothpp}}  \le  \frac{\epsilon}{C} \Big( 1+
    \sum_{m = 1 }^\infty (m+1)^{3/2-{\smoothpp}} \Big) < \epsilon.
\end{eqnarray*}
 In the second inequality we used \eqref{varphibound} and in the last step $\smoothpp>5/2$
   in order to guarantee the convergence of the series. It is notable that the estimate \eqref{cof} already implies
   $$\big| \left<\funk,   \varphi_{(l, m)}\right>_{\mathcal{L}^2(\mathcal{B}, \mu)} -
    \left<\tilde{g},  \varphi_{(l, m)}\right>_{\mathcal{L}^2(\mathcal{B}, \mu)}
    \big| \le \frac{\epsilon}{C (m+1)^{\smoothpp}},$$
 for all $m \ge (C/\epsilon)^{1/(\smoothp-\smoothpp)}$ i.e. substantially different
 coefficients can only occur for smaller $m$.

Now let us consider those coefficients with
 $m \le \lceil (C/\epsilon)^{1/(\smoothp-\smoothpp)}\rceil$. In order to
 construct the desired functions for a covering of $\smoothel2$ as in \eqref{covering},
 we decompose the domains of possible Radon coefficients in the following way:  For each
 $(l, m) \in \mathcal{N}$, the estimate \eqref{cof} implies that
    $$ \lbp  \funk, \varphi_{(l,m)}
    \rbp_{\mathcal{L}^2(\mathcal{B}, \mu)} \in \left[-(m+1)^{-\smoothp}, (m+1)^{-\smoothp}
     \right] \times \left[-i(m+1)^{-\smoothp}, i(m+1)^{-\smoothp}  \right].$$
 We can introduce $ \lceil 4 C  (m+1)^{\smoothp-\smoothpp}/ \epsilon  \rceil^2$ grid
 points to this cube, such that any two of them have maximal distance
 $\epsilon/(C (m+1)^{\smoothpp})$. 
 The set of grid points for each cube will be called $G_{(l, m)}$.
 It then follows that for each function $g$ in $\smoothel2$ we
 can find a vector of coefficients
    $$\mathbf{a}:=\lb a_{(l,m)}\rb \in \times_{m=0}^{\lceil
    (C /\epsilon)^{1/(\smoothp-\smoothpp)} \big\rceil} \times_{l=-m}^{\bullet m}
    G_{(l,m)},$$
 ($\bullet$ denotes multiplications with those indices only where $m-l$ is even) such
 that the corresponding function
    $$  \tilde{g}= \sum_{m = 0}^{\lceil
    (C /\epsilon)^{1/(\smoothp-\smoothpp)} \big\rceil} \sum_{l=-m}^{m}
    \varphi_{(l,m)}a_{(l,m)}$$
 satisfies \eqref{16}  and hence has maximal distance $\epsilon$
 to $g$. Here the coefficients $a_{(l,m)}$ for $m-l$ odd are simply assumed to be $0$. The covering number will hence be bounded by the total number of such
 coefficients, which can be calculated as follows:
\begin{eqnarray}
 & & \label{cont} \Big| \times_{m=0}^{\lceil
    (C /\epsilon)^{1/(\smoothp-\smoothpp)} \rceil} \times_{l=-m}^{\bullet m}
     G_{(l,m)} \Big| = \Big| G_{(0,0)} \Big|\prod_{m=1}^{\lceil     (C /\epsilon)
     ^{1/(\smoothp-\smoothpp)} \rceil} \prod_{l=-m}^{\bullet m} \Big| G_{(l,m)} \Big|
     \\[ 1ex]
    & \le &\Big\lceil\frac{4C}{\epsilon }\Big\rceil^2 \prod_{m=1}
    ^{  \lceil \lb  C /\epsilon\rb^{1/(\smoothp-\smoothpp)} \rceil} \Big\lceil
    \lb \frac{4 C  m^{\smoothpp-\smoothp} }{\epsilon}\rb \Big\rceil^{2(m+1)} \nonumber
    \le  \Big\lceil\frac{4C}{\epsilon }\Big\rceil^2 \prod_{m=1}^{  \lceil \lb  C
    /\epsilon\rb^{{1/(\smoothp-\smoothpp)}} \rceil} \Big\lceil \lb \frac{4 C   }
    {\epsilon}\rb \Big\rceil^{2(m+1)}   \\[1ex]
    & = & \Big\lceil \frac{C}{\epsilon}\Big\rceil^{2\sum_{m=1}^{\lceil \lb
    C /\epsilon\rb^{1/(\smoothp-\smoothpp)} \rceil+1}m} \nonumber \le \Big\lceil
     \frac{C }{\epsilon}\Big\rceil^{8 \lb  C /\epsilon\rb^{2/(\smoothp-\smoothpp)}}
    \le \exp \Big\{ \log\Big( \frac{C}{\epsilon}\Big) \Big(  \frac{C}{\epsilon}\Big)^{2/(\tilde{\tau}-
     \tilde{t})} \Big\}.     \nonumber
\end{eqnarray}
 To achieve the desired rate we repeat our above argumentation
 for a shrunk version of $t$, say $t-\delta$ which is still larger than $3$, i.e. with
 $\tilde{t}-\delta$ still larger than $5/2$. For sufficiently small $\epsilon>0$ it follows that
    $$\exp \Big\{ \log\Big( \frac{C}{\epsilon}\Big) \Big(  \frac{C}{\epsilon}\Big)^{2/(\tilde{\tau}
    -\tilde{t}+\delta)} \Big\} \le \exp \Big\{ \Big(  \frac{C}{\epsilon}\Big)^{2/(\tilde{\tau}-
    \tilde{t})} \Big\}.$$

 By our auxiliary considerations the bracketing number is thus bounded in the desired way.\qed

\end{document}